\documentclass[11pt]{amsart}
\usepackage[utf8]{inputenc}

\usepackage[nospace,noadjust]{cite}
\usepackage{relsize}
\usepackage{cite}
\usepackage{color}
\usepackage[english]{babel}
\usepackage{comment}

\renewcommand{\emph}[1]{\textbf{#1}}

\usepackage{tikz-cd}
\usetikzlibrary{calc,fit,matrix,arrows,automata,positioning}
\usetikzlibrary{decorations.pathreplacing,angles,quotes}
\usepackage{colortbl}
\usepackage[dvipsnames]{xcolor}
\usepackage{hyperref, enumitem}
\hypersetup{
  colorlinks   = true, 
  urlcolor     = blue, 
  linkcolor    = Purple, 
  citecolor   = red 
}
\usepackage{amsmath,amsthm,amssymb}

\usepackage[margin=2.5cm]{geometry}

\usepackage{stmaryrd}

\usepackage{array}
\newcolumntype{x}[1]{>{\centering\arraybackslash\hspace{0pt}}p{#1}}

\theoremstyle{definition}
\newtheorem{theorem}{Theorem}[section]
\newtheorem{definition}[theorem]{{{Definition}}}
\newtheorem{example}[theorem]{{{Example}}}
\newtheorem{notation}[theorem]{{{Notation}}}
\newtheorem{remark}[theorem]{{{Remark}}}

\newtheorem{corollary}[theorem]{{{Corollary}}}
\newtheorem{proposition}[theorem]{{{Proposition}}}
\newtheorem{lemma}[theorem]{{{Lemma}}}

\usepackage{mathrsfs}

\DeclareMathOperator{\ccc}{c}
\DeclareMathOperator{\rrr}{r}

\newcommand{\numberset}{\mathbb}
\newcommand{\N}{\numberset{N}}
\newcommand{\FLH}{\mathscr{F}_{\rm LH}(n,k)}
\newcommand{\FAH}{\mathscr{F}_{\rm AH}(n,k)}
\newcommand{\FLR}{\mathscr{F}_{\rm LR}(n,k)}

\newcommand{\FMR}{\mathscr{F}_{\rm MR}(m,n,k)}
\newcommand{\FMRc}{\mathscr{F}_{{\rm MR}_{\ccc}}(m,n,k)}
\newcommand{\FMRr}{\mathscr{F}_{{\rm MR}_{\rrr}}(m,n,k)}

\newcommand{\Fast}{\mathscr{F}_\ast(n,k)}
\newcommand{\GLH}{{\rm G}^{\rm LH}_{m,n}}
\newcommand{\GAH}{{\rm G}^{\rm AH}_{m,n}}
\newcommand{\GLR}{{\rm G}^{\rm LR}_{m,n}}

\newcommand{\GMR}{{\rm G}^{\rm MR}_{m,n}}
\newcommand{\Gast}{{\rm G}^\ast_{m,n}}
\newcommand{\SLH}{\mathscr{S}_{\rm LH}(n,k)}
\newcommand{\SAH}{\mathscr{S}_{\rm AH}(n,k)}
\newcommand{\SLR}{\mathscr{S}_{\rm LR}(n,k)}

\newcommand{\SMRr}{\mathscr{S}_{\rm MR_{\rrr}}(m,n,k)}
\newcommand{\SMRc}{\mathscr{S}_{\rm MR_{\ccc}}(m,n,k)}

\newcommand{\Sast}{\mathscr{S}_{\ast}(n,k)}
\newcommand{\sLH}{\mathcal{S}}
\newcommand{\sAH}{\mathcal{S}}
\newcommand{\sLR}{\mathcal{S}}
\newcommand{\sAR}{\mathcal{S}}

\newcommand{\sMR}{\mathcal{S}}

\newcommand{\PP}{\numberset{P}}
\newcommand{\C}{\mathcal{C}}
\newcommand{\K}{\numberset{K}}
\newcommand{\F}{\numberset{F}}

\newcommand{\mS}{\mathcal{S}}
\newcommand{\mT}{\mathcal{T}}
\newcommand{\mC}{\mathcal{C}}
\newcommand{\mP}{\mathcal{P}}

\newcommand{\mm}{\boldsymbol{\rm m}}

\newcommand{\mD}{\mathcal{D}}
\newcommand{\mU}{\mathcal{U}}

\newcommand{\mX}{\mathcal{X}}
\newcommand{\mY}{\mathcal{Y}}
\newcommand{\mV}{\mathcal{V}}

\newcommand{\mR}{\mathcal{R}}

\newcommand{\mZ}{\mathcal{Z}}

\newcommand{\mH}{\mathcal{H}}

\newcommand{\wt}{\textnormal{wt}}
\newcommand{\cwt}{\textnormal{cwt}}
\newcommand{\rwt}{\textnormal{rwt}}

\newcommand{\Fq}{\F_q}

\newcommand{\colsp}{\textnormal{colsp}}

\newcommand{\rowsp}{\textnormal{rowsp}}

\newcommand{\HH}{\textnormal{H}}
\newcommand{\rk}{\textnormal{rk}}

\DeclareMathOperator{\GL}{GL}
\DeclareMathOperator{\supp}{supp}

\DeclareMathOperator{\PGL}{PGL}

\DeclareMathOperator{\sss}{ss}
\DeclareMathOperator{\dd}{d}
\DeclareMathOperator{\csupp}{csupp}
\DeclareMathOperator{\rsupp}{rsupp}

\title{The geometry of rank-metric codes}
\usepackage[foot]{amsaddr}
\author{Gianira N. Alfarano$^1$}
\author{Martino Borello$^2$}
\author{Alessandro Neri$^3$}
\address{$^1$Universit\'e de Rennes, IRMAR, Campus de Beaulieu, F-35042 Rennes Cedex, France.}
\address{$^2$Universit\'e Paris 8, Laboratoire de G\'eom\'etrie, Analyse et Applications, LAGA, Universit\'e Sorbonne Paris Nord, CNRS, UMR 7539, F-93430, Villetaneuse, France.
}
\address{$^3$Department of Mathematics and Applications “R. Caccioppoli”, University of Naples Federico II, Via Cintia,
Monte S. Angelo, 80126 Napoli, Italy.}

\email{gianira-nicoletta.alfarano@univ-rennes.fr}
\email{martino.borello@univ-paris8.fr}
\email{alessandro.neri@unina.it}

\begin{document}
\begin{abstract}
In this paper, we develop a geometric framework for matrix rank-metric codes based on generator tensors and their slice spaces. To every nondegenerate matrix rank-metric code, we associate two systems, which translate metric properties of the code into geometric conditions involving intersections with hyperplanes. This leads to a correspondence between equivalence classes of nondegenerate matrix rank-metric codes and equivalence classes of systems, as well as to Delsarte-type incidence identities relating the rank distribution of a code over a finite field to those of its associated systems. As an application, we introduce generalized weights through
the notion of evasive systems, study faithful and one-weight codes over finite fields, and recover known bounds and results from the theory of semifields. Finally, we use this framework to associate additive Hamming-metric codes with matrix rank-metric codes and show that several metric properties are preserved under this correspondence.
\end{abstract}

\maketitle

\section*{Introduction}

The interplay between coding theory and finite geometry has repeatedly shown that the metric properties of a code can often be better understood through the geometry underlying it. Since the foundational correspondence between linear codes endowed with the Hamming metric and projective systems \cite{tsfasman2013algebraic}, geometric methods have provided a structural language in which many coding-theoretic properties become transparent. Parameters such as minimum distance, generalized weights, and weight distribution admit natural interpretations in terms of incidence relations with hyperplanes and higher-dimensional subspaces, while classification problems are often translated into the study of highly constrained point configurations in projective space. A classical example is the equivalence between MDS codes and arcs in projective spaces; see \cite{segre1955curve}.

In recent years, rank-metric codes have also attracted considerable attention, owing both to their rich algebraic structure and to their applications in network coding, distributed storage, cryptography, and related areas. Introduced independently by Delsarte \cite{de78} and Gabidulin \cite{gabidulin1985theory}, these codes replace the Hamming weight with the rank of vectors or matrices. In the vectorial setting, their geometry is now well understood. Indeed, for linear rank-metric codes the correspondence with $q$-systems has established a geometric framework parallel to that of projective systems in the Hamming-metric case; see \cite{sheekey2019scatterd, randrianarisoa2020geometric}. Within this framework, notions such as scatteredness, evasiveness, and generalized rank weights acquire geometric meaning, and several structural and classification results can be derived from geometric arguments; see, for instance, \cite{alfarano2021linear, marino2023evasive}.

A notable illustration of the effectiveness of this geometric viewpoint is the classification of one-weight codes. In the Hamming metric, Bonisoli characterized one-weight codes in geometric terms \cite{bonisoli1984every}; in the rank metric, an analogous characterization was obtained by Randrianarisoa \cite{randrianarisoa2020geometric}. In both settings, one-weight codes correspond to highly regular configurations in projective spaces, revealing the combinatorial symmetry underlying their algebraic structure.

However, the broadest case in rank-metric theory has not yet been equipped with a comparable intrinsic geometric model. Matrix rank-metric codes, namely subspaces of matrices endowed with the rank distance, constitute the natural and most general setting for the rank metric. They arise in Delsarte's original formulation \cite{de78}, encompass vector rank-metric codes after field expansion, and provide the native framework for many extremal and structural questions in the theory. Nevertheless, unlike their Hamming-metric and vectorial rank-metric counterparts, matrix rank-metric codes have so far lacked an analogue of projective systems or $q$-systems capable of translating their algebraic and metric properties into geometric language. Without such a correspondence, matrix rank-metric codes remain partly detached from the geometric methods that have proved so powerful in other areas of coding theory. Questions concerning generalized weights, one-weight structures, extremal configurations, and classification phenomena can certainly be approached algebraically, but they do not benefit from the conceptual clarity and structural insight typically provided by geometry. 

\medskip

\textbf{Our contribution. } The main purpose of this paper is to provide an intrinsic geometric framework
for matrix rank-metric codes. While projective systems and \(q\)-systems give
geometric models for Hamming-metric codes and vector rank-metric codes,
respectively, no comparable model has been available for general matrix
rank-metric codes. We fill this gap using the tensor representation of
matrix codes introduced in~\cite{byrne2019tensor}. Given a generator tensor of a matrix
rank-metric code, one of its slice spaces is the code itself, while the other two give rise to the \emph{column-} and \emph{row-systems associated with the code}. These systems translate metric properties of the original code into incidence properties with hyperplanes, in direct analogy with the classical correspondence between linear Hamming-metric codes and projective systems, and
with the correspondence between vector rank-metric codes and $q$-systems. Thus, tensor slice spaces provide the correct geometric objects for matrix
rank-metric codes.

This provides a unified geometric framework for studying linear codes endowed with either the Hamming metric or the rank metric. Our first main result is a correspondence theorem: equivalence classes of column-nondegenerate matrix rank-metric codes are in one-to-one correspondence with equivalence classes of column-systems, and analogously for row-systems; see Theorem~\ref{thm:correspondence}.
This shows that the systems introduced here are not auxiliary constructions,
but genuine geometric representatives of matrix rank-metric codes.
We point out that, independently of the present work, Byrne and Sheekey obtained the same definition of a system; see \cite{byrne2026bounds}. Apart from the introduction of this common object, however, the two works pursue substantially different directions.
The strength of this viewpoint already appears over finite fields, where we
derive Delsarte-type incidence identities relating the rank distribution of a
matrix rank-metric code to that of an associated system. The first such
identity, Theorem~\ref{thm:StandardEq}, should be viewed as the matrix-rank
analogue of the first standard equation for projective systems: it counts, in
two ways, incidences between nonzero elements of a system and hyperplanes of
the ambient space. More generally, by replacing hyperplanes with subspaces of
fixed codimension, we obtain higher-order identities; see
Theorem~\ref{thm:higher-delsarte-identities}. These identities provide a
bridge between the metric distribution of the code and the geometry of its
associated systems.

The geometric framework also naturally leads to a notion of \textit{evasiveness} for matrix systems. This generalizes the role played by scattered and evasive
$q$-systems in the vector rank-metric setting. We show that generalized
rank weights of matrix codes (Definition \ref{def:gen_weights_nostra}) can be read from the evasiveness properties of
their associated column- and row-systems. In particular, classical bounds for
generalized weights acquire a transparent geometric proof, and equality is
characterized by scatteredness-type conditions on the associated systems.

We then use our Delsarte-type identities to study faithful and one-weight matrix
rank-metric codes. Faithful codes may be viewed as the geometric dual
counterpart of codes whose nonzero codewords all have maximum possible rank. We show that column-faithfulness and row-faithfulness are, in fact, equivalent,
so that faithfulness is an intrinsic property of the code; see Theorem~\ref{thm:faithfulness}. As a consequence, every faithful matrix rank-metric code over a finite field contains a codeword of maximum possible rank. More precisely, we derive an explicit lower bound on the number of codewords of rank \(\min\{m,n\}\), which improves the basic
existence statement; see Theorem~\ref{thm:Amin2}.
For one-weight codes, the same incidence identity gives a short geometric
proof of the Singleton-like bound of \cite{dumas2010subspaces} on their parameters; see Theorem \ref{thm:bound_one_weight}. In the
extremal case, the correspondence also recovers Knuth's classical operations on finite semifields (\cite{knuth1965finite}): passing to the associated column- and row-systems corresponds, together with transposition, to the transformations generating
the Knuth orbit; see Corollary \ref{cor:knuth}.

Finally, we associate additive Hamming-metric codes with matrix rank-metric
codes by projectivizing the nonzero elements of an associated system. The
resulting additive codes retain precise information about the original
rank-metric code: their Hamming weights and generalized Hamming weights are
explicit functions of the rank weights and generalized rank weights of the
matrix code.

\medskip 

\textbf{Outline. }The paper is organized as follows. In Section \ref{sec:background}, we collect the notation and preliminary material needed throughout the paper. In Section \ref{sec:geometry}, we recall the geometric viewpoint on linear and additive Hamming-metric codes, as well as on vector rank-metric codes. In Section \ref{sec:q-system-matrices}, we introduce the column- and row-systems associated with a
matrix rank-metric code and establish a correspondence between equivalence
classes of nondegenerate matrix rank-metric codes and equivalence classes of
systems. We also prove Delsarte-type incidence identities relating the rank
distribution of a code to those of its associated systems. Section \ref{sec:bounds} is devoted to evasive systems, generalized weights, and related bounds. In Section \ref{sec:conseq}, we specialize to finite fields and study the consequences of these identities for faithful and one-weight codes. Finally, in Section \ref{sec:extended-Hamming}, we explain how to associate an additive Hamming-metric code with a matrix rank-metric code and analyze how metric properties are transferred under this correspondence.

\section*{Acknowledgments} 
This research has been partially supported by the Italian National Group for Algebraic and Geometric Structures and their Applications (GNSAGA - INdAM) and by Università Italo Francese (UIF/UFI) via PHC Galileo 2026 - G26-260/54322VM.
G. N. Alfarano is supported by the Agence Nationale de la Recherche through grant number ANR-24-CPJ1-0075-01. M. Borello is partially
supported by the ANR-21-CE39-0009 - BARRACUDA (French Agence Nationale de la Recherche). A. Neri is supported by the INdAM - GNSAGA Project CUP E53C24001950001  ``Noncommutative polynomials in coding theory''. This work has been partially written during the Opera 2026 conference in Bordeaux and so we acknowledge the support from
the International Research Laboratory LYSM in partnership between CNRS and INdAM.

\section{Background and notations}\label{sec:background}

Let $m$ be a positive integer. Throughout the paper, $$\boxed{\K \text{ is a field and }\F\text{ is an extension of }\K\text{ of degree 
} [\F:\K]=m.}$$
$\K$ and $\F$ are fixed, and we will indicate, where appropriate, whether they are to be considered finite.
We use the following notations/conventions.

\begin{notation}
    We denote by $\K^{N_1\times N_2}$ the space of $N_1\times N_2$ matrices over $\K$ for some positive integers $N_1,N_2$. For a matrix $M\in\K^{N_1\times N_2}$, we denote by $\rowsp_\K(M)$ and $\colsp_\K(M)$, the $\K$-spaces generated by the rows and the columns of $M$ respectively. In particular, $\rowsp_\K(M)\subseteq \K^{N_2}$ and $\colsp_\K(M)\subseteq\K^{N_1}$. 
    For $\mS\subseteq\K^N$, $\mS^\perp:=\{v\in \K^N:\langle v,s\rangle=0\}$, where $\langle\cdot ,\cdot \rangle:\K^N\times \K^N\to \K$ is the standard inner product. Moreover, for $u\in\K^N$, we denote by $u^\perp$ the subspace of $\K^N$ which is orthogonal (with respect to the standard inner product) to $\langle u\rangle_\K$. All vectors are considered row vectors. We write $[N]:=\{1,\ldots,N\}$.
    For any space $V$, we denote by $\mathbb{P}(V)$, its corresponding projective space. Throughout the paper, $k,n$ are positive integers.
\end{notation}

In the rest of the paper, we are interested in tensor products of the form $\K^{N_1}\otimes \K^{N_2}\otimes \K^{N_3}$, whose elements are called $3$-tensors, third-order tensors, or triads. Such elements may be represented
as three-dimensional arrays, and we identify $\K^{N_1}\otimes \K^{N_2}\otimes \K^{N_3}$ with
$\K^{N_1\times N_2\times N_3}$. We also use the natural identifications
$\K^{1\times N}=\K^N$ and
$\K^{1\times N_1\times N_2}
=\K^{N_1\times 1\times N_2}
=\K^{N_1\times N_2\times 1}
=\K^{N_1\times N_2}$.
We now recall some definitions and results from tensor algebra. The interested reader is referred to \cite{cooperstein2010advanced} for more details.
We introduce the following maps, which define multiplication
of $3$-tensors with vectors and matrices:
\begin{align*}
    m_1: \K^{N_1\times N_2\times N_3} \times \K^{N_1\times s} &\to \K^{s\times N_2\times N_3} : (T,A_1) \mapsto m_1(T, A_1) = \sum_i (u_iA_1) \otimes v_i \otimes w_i,\\
    m_2: \K^{N_1\times N_2\times N_3}\times \K^{N_2\times s} &\to \K^{N_1\times s\times N_3} : (T,A_2) \mapsto m_2(T, A_2) = \sum_i  u_i\otimes  (v_iA_2) \otimes   w_i,\\
    m_3: \K^{N_1\times N_2\times N_3} \times \K^{N_3\times s} &\to \K^{ N_1\times N_2 \times s} : (T,A_3) \mapsto m_3(T, A_3) = \sum_i  u_i \otimes v_i \otimes(w_iA_3),
\end{align*}
for any $T =\sum_i u_i \otimes v_i \otimes w_i \in\K^{N_1\times N_2 \times N_3}$.

\begin{definition}
    Let $T\in\K^{N_1\times N_2 \times N_3}$. For each $i\in\{1, 2, 3\}$, we define the \textbf{$i$-th slice space} of $T$ to be the $\K$-span of $\{m_i(T,e_j) : 1 \leq j \leq N_i\}$, that is,
$$\sss_i(T) := \langle m_i(T,e_1),\ldots, m_i(T,e_{N_i})\rangle_\K,$$
where $e_j$ denotes the $j$-th standard basis vector.
\end{definition}


\section{The geometry of linear, additive, and matrix codes}\label{sec:geometry}


In this section, we want to guide the reader through a journey into a known geometric perspective of linear, additive, and matrix codes. We will provide a unified point of view, explaining the settings in which the geometric perspective shed new light on the properties of codes. Moreover, we will explore the interplay of all these families of codes. Recall that we have fixed the degree $m$ field extension $\F/\K$.

\subsection{Linear, additive, and matrix codes}
Let $n$ be a positive integer.
A vector $v=(v_1,\ldots,v_n)\in\F^n$ can be identified with a matrix $\Gamma(v)\in\K^{m\times n}$, obtained by expanding the coefficients of $v$ with respect to a basis $\Gamma=\{\gamma_1,\ldots,\gamma_m\}$ of $\F/\K$, i.e. if $v_j=\sum_{i=1}^m v_{i,j} \gamma_i$ for each $j\in[n]$, then $\Gamma(v)=(v_{i,j})_{i,j}$. It is straightforward to verify that if $\Gamma'$ is another basis of $\F/\K$, we have $\rk(\Gamma(v))=\rk(\Gamma'(v))$.

\medskip

We can endow $\F^n$ with two metrics. The first one is the \textbf{Hamming metric}, induced by the \textbf{Hamming weight}
$$\wt_{\rm H}:\begin{array}{ccc}
\F^n &\to &\N_0   \\
v &\mapsto & |\{i\in[n] \; : \; v_i\neq 0\}|
\end{array}.$$
The corresponding distance between two vectors $v,w\in\F^n$ is given by $\dd_{\rm H}(v,w)=\wt_{\rm H}(v-w)$. The set $\{i\in [n] \; : \; v_i\neq0\}$ is called the \emph{Hamming support} of $v$, denoted by $\supp_{\rm H}(v)$.

The second is the \emph{rank metric} over $\K$, induced by the \emph{rank weight} 
$$\wt_{\rk}:\begin{array}{ccc}
\F^n &\to &\N_0   \\
v &\mapsto & \rk(\Gamma(v))
\end{array}.$$
The corresponding distance between two vectors $v,w\in\F^n$ is given by $\dd_{\rk}(v,w)=\wt_{\rk}(v-w)$. The subspace
$\rowsp_\K(\Gamma(v))$ is called the \emph{rank support} of $v$ and is denoted by
$\supp_{\rk}(v)$. It does not depend on the choice of $\Gamma$; see, e.g.,
\cite{alfarano2021linear}.\footnote{It is also possible to define the rank support of
$v\in\F^n$ as $\colsp_\K(\Gamma(v))$ or, equivalently, as
$\langle v_1,\ldots,v_n\rangle_\K$. However, while
$\rowsp_\K(\Gamma(\lambda v))=\rowsp_\K(\Gamma(v))$ for every
$\lambda\in\F^\ast$ and every choice of basis $\Gamma$, in general
$\colsp_\K(\Gamma(\lambda v))\neq \colsp_\K(\Gamma(v))$. In the geometric setting
developed here, it is desirable that the support of a codeword be invariant under
multiplication by scalars.}

We can also endow $\K^{m \times n}$ with three metrics. The first is the \textbf{column Hamming metric}, induced by the \textbf{column Hamming weight}
$$\wt^{\ccc}_{\rm H}:\begin{array}{ccc}
\K^{m \times n} &\to &\N_0   \\ 
M &\mapsto & |\{j \in [n] \; : \;  M^{(j)} \neq 0\}|\end{array},$$
where $M^{(j)}$ denotes the $j$-th column of $M$. The corresponding distance
between two matrices $M,N\in\K^{m\times n}$ is given by $\dd^{\ccc}_{\rm H}(M,N) = \wt^{\ccc}_{\rm H}(M-N)$. The set $\{j \in [n] : M^{(j)} \neq 0\}$ is called the \textbf{column Hamming support} of $M$, denoted by $\supp^{\ccc}_{\rm H}(M)$. 

The second is the \textbf{row Hamming metric}, induced by the \textbf{row Hamming weight}
$$\wt^{\rrr}_{\rm H}:\begin{array}{ccc}
\K^{m \times n} &\to &\N_0   \\ 
M &\mapsto & |\{i \in [m] \; : \;  M_{(i)} \neq 0\}|\end{array},$$
where $M_{(i)}$ denotes the $i$-th row of $M$.  The corresponding distance is $\dd^{\rrr}_{\rm H}(M,N) = \wt^{\rrr}_{\rm H}(M-N)$. The set $\{i \in [m] : M_{(i)} \neq 0\}$ is called the \textbf{row Hamming support} of $M$, denoted by $\supp^{\rrr}_{\rm H}(M)$.

The third is the \textbf{rank metric}, induced by the usual rank of matrices,
$${\rk}:\begin{array}{ccc}
\K^{m \times n} &\to &\N_0, \\
M &\mapsto & \rk(M)\end{array}.$$
The corresponding distance is  $\dd_{\rk}(M,N) = {\rk}(M-N)$. The subspace
$\rowsp_\K(M)$ is called the \textbf{row support} of $M$ and is denoted by
$\mathrm{rsupp}_{\rk}(M)$, while the subspace $\colsp_\K(M)$ is called the
\textbf{column support} of $M$ and is denoted by
$\mathrm{csupp}_{\rk}(M)$.

In this paper, we consider five families of \textbf{codes}\footnote{In the literature, the term ``code'' often refers more broadly to nonlinear codes, namely subsets of the ambient space without a prescribed algebraic structure. However, in this paper all codes are assumed to be linear unless stated otherwise; we therefore use the terms ``code'' and ``subspace'' interchangeably.}: 
\begin{itemize}\setlength{\itemsep}{0.8em}
    \item[(LH)] $\FLH:=\{$nondegenerate $k$-dimensional  $\F$-subspaces of $\F^n$ endowed with the Hamming metric$\}$. We refer to codes in $\FLH$ as \textbf{Hamming-metric codes}. 
    \item[(AH)] $\FAH:=\{$nondegenerate $k$-dimensional $\K$-subspaces of $\F^n$ endowed with the Hamming metric$\}$. We refer to codes in $\FAH$ as \textbf{(Hamming-metric) additive codes}.
    \item[(LR)] $\FLR:=\{$nondegenerate $k$-dimensional $\F$-subspaces of $\F^n$ endowed with the rank metric$\}$. We refer to codes in $\FLR$ as \textbf{$\F$-linear rank-metric codes} or \textbf{vector rank-metric codes}.
\item[(MR)] $\FMR=\FMRc\ \cup \ \FMRr$, where $\FMRc:=\{$column-non\-degene\-rate $k$-dimensional $\K$-subspaces of $\K^{m\times n}$ endowed \ with the rank metric$\}$ \ and $\FMRr:=\{$row-nondegenerate $k$-dimensional $\K$-subspaces of $\K^{m\times n}$ endowed with the rank metric$\}$. We refer to codes in $\FMR$ as \textbf{($\K$-linear) rank-metric codes} or \textbf{matrix rank-metric codes}. 
\end{itemize}

Here, ``L'' stands for linear, ``A'' for additive, and ``M'' for matrix code, while the second letter indicates the metric. The notion of \textit{degeneracy} depends on the family under consideration. In the cases (LH) and (AH), a code is \textbf{degenerate} if there exists a coordinate in which all its elements are identically zero. In the case (LR), a code is \textbf{degenerate} if there exists a matrix $A\in\GL_n(\K)$ such that, after multiplying all codewords on the right by $A$, one coordinate is identically zero.
For codes in $\K^{m\times n}$, we distinguish between column and row degeneracy, leading to the two families (MR$_{\ccc}$) and (MR$_{\rrr}$). A code is \emph{column-degenerate} if there exist $A\in\GL_n(\K)$ and $j\in[n]$ such that, for every codeword $M$, the $j$-th column of $MA$ is zero. Similarly, a code is \emph{row-degenerate} if there exist $B\in\GL_m(\K)$ and $i\in[m]$ such that, for every codeword $M$, the $i$-th row of $BM$ is zero. In all cases, a code is called \emph{nondegenerate} if it is not degenerate. Equivalently, a code is degenerate if it can be isometrically embedded into an ambient space of smaller dimension.

\begin{definition}\label{def:codes&supp}
We define the parameters of the codes in these families as follows.
\begin{itemize}\setlength\itemsep{0.8em}
    \item[(LH)] An $[n,k,d]_\F$ code $\mC$ in $\FLH$ is a code of \emph{length} $n$, $\F$-\emph{dimension} $k$, and \emph{minimum distance} $d=\min_{c\in \mC\setminus\{0\}}\wt_{\HH}(c)$. We define the \textbf{support} of $\mC$ as $\supp_{\rm H}(\mC) = \bigcup_{c\in\mC}\supp_{\rm H}(c)\subseteq [n]$.
    
    \item[(AH)] An $[n,k/m,d]_\K^\F$ code $\mC$ in $\FAH$ is a code of \emph{length} $n$, $\K$-\emph{dimension} $k$, and \emph{minimum distance} $d=\min_{c\in \mC\setminus\{0\}}\wt_\HH(c)$. We define the \textbf{support} of $\mC$ as $\supp_{\rm H}(\mC) = \bigcup_{c\in\mC}\supp_{\rm H}(c)\subseteq [n]$.

    \item[(LR)] An $[n,k,d]_{\F/\K}$ code $\mC$ in $\FLR$ is a code of \emph{length} $n$, $\F$-\emph{dimension} $k$, and \emph{minimum distance} $d=\min_{c\in \mC\setminus\{0\}}\wt_{\rk}(c)$. We define the \textbf{support} of $\mC$ as $\supp_\rk(\mC) = \sum_{c\in\mC}\mathrm{supp}_{\rk}(c)\subseteq \K^n$.
    \item[(MR)] An $[m\times n,k,d]_\K$ code $\mC$ in $\FMR$ is a code of $\K$-\emph{dimension} $k$, and \emph{minimum distance} $d=\min_{M\in \mC\setminus\{0\}}\rk(M)$. We define the \textbf{row-support} of $\mC$ and the \textbf{column-support} of $\mC$ as $\mathrm{rsupp}_\rk(\mC) = \sum_{M\in\mC}\mathrm{rsupp}_{\rk}(M)\subseteq \K^n$ and $\mathrm{csupp}_\rk(\mC) = \sum_{M\in\mC}\mathrm{csupp}_{\rk}(M)\subseteq \K^m$, respectively. 
\end{itemize}

\end{definition}

When the minimum distance is not relevant, we omit the parameter $d$ from the
notation introduced above.

\begin{remark}
When the underlying fields are finite, it is customary in the notation for codes to indicate the size of the field, or the degree of the field extension, rather than the field itself.
\end{remark}

\begin{remark}\label{rem:non-degeneracy-space}
Definition \ref{def:codes&supp} implies the following characterizations of
nondegeneracy.
\begin{enumerate}
    \item[(a)] A code $\mC\in\{\FLH,\FAH\}$ is nondegenerate if and only if
    $\supp_{\rm H}(\mC)=[n]$.

    \item[(b)] A code $\mC\in\FLR$ is nondegenerate if and only if
    $\supp_{\rk}(\mC)=\K^n$.

    \item[(c)] A code $\mC\in\FMR$ is column-nondegenerate, respectively
    row-nondegenerate, if and only if
    $\mathrm{rsupp}_{\rk}(\mC)=\K^n$, respectively
    $\mathrm{csupp}_{\rk}(\mC)=\K^m$.
\end{enumerate}
Throughout the paper, we often assume nondegeneracy. However, the theory extends
to degenerate codes in the same way, by isometrically embedding them into smaller
ambient spaces in which they become nondegenerate.
\end{remark}

Since all codes considered in this paper are subspaces, they can be described compactly by matrices and/or tensors.
A \emph{generator matrix} for $\mC\in \{\FLH,\FLR\}$ is a matrix $G\in \F^{k\times n}$ such that $\mC=\rowsp_\F(G)=\{uG \; : \; u\in\F^k\}$. A \emph{generator matrix} for $\mC\in \FAH$ is a matrix $G\in \F^{k\times n}$ such that $\mC=\rowsp_\K(G)=\{uG \; : \; u\in\K^k\}$.  A \emph{generator tensor} for $\mC\in \FMR$ is a tensor $T\in \K^{k\times m\times n}$ such that $\mC=\sss_1(T)=\{m_1(T,u) \; : \; u\in\K^k\}$.

We consider the following groups acting on $\F^n$.

\begin{itemize}
    \item[(LH)] The group
    $\GLH=(\F^\ast)^n\rtimes \mathrm{Sym}(n)$, where $\mathrm{Sym}(n)$ is the
    symmetric group on $n$ elements, acts by permuting the coordinates and
    multiplying each coordinate by an element of $\F^\ast$. Equivalently, every
    element of $\GLH$ can be represented by an $n\times n$ monomial matrix over
    $\F$, i.e., a matrix with exactly one nonzero entry in each row and each
    column. The action on $\F^n$ is given by multiplication on the right.
    \item[(AH)] The group
    $\GAH=\GL_m(\K)^n\rtimes \mathrm{Sym}(n)$ acts by applying an element of
    $\GL_m(\K)$ to each coordinate of a vector, viewed as an element of the
    $\K$-vector space $\F$, and by permuting the coordinates. The action of
    $\GL_m(\K)$ on each coordinate is the natural one\footnote{This action
    depends on the choice of a basis of $\F$ over $\K$; by ``the'' action we
    mean any of the actions arising in this way.} by multiplication on the
    right.

    \item[(LR)] The group $\GLR:=\GL_n(\K)$ acting by multiplication on the right.
\end{itemize}

We also consider the following group acting on $\K^{m\times n}$.

\begin{itemize}
    \item[(MR)] The group $\GMR:=\GL_m(\K)\times\GL_n(\K)$, where, for
    $M\in\K^{m\times n}$, the right action of $(A,B)\in\GMR$ on $M$ is defined by $M\cdot (A,B)=A^\top MB$.
\end{itemize}

\begin{remark}
The actions of the groups defined above are isometries with respect to the corresponding metrics: the Hamming metric in the cases (LH) and (AH), and the rank metric in the cases (LR) and (MR). Moreover, these actions are linear over the corresponding field, namely over $\F$ in the cases (LH) and (LR), and over $\K$ in the cases (AH) and (MR).
It is straightforward to check that these actions give all linear isometries of the corresponding ambient spaces, with respect to the relevant fields and metrics, except in the following cases. For the family $\FLR$, one may also multiply by nonzero scalars of $\F$; for the family $\FMR$, when $m=n$, one may
also consider transposition. These additional isometries will not be considered
in this paper. 
\end{remark}

\begin{definition}
Let $\ast\in\{{\rm LH},{\rm AH},{\rm LR}\}$. Two codes
$\mC_1,\mC_2\in\Fast$ are said to be  \emph{equivalent} if there exists
$g\in\Gast$ such that $\mC_1\cdot g:=\{c\cdot g \; : \; c\in \mC_1\}= \mC_2$, where $\cdot$ denotes the corresponding action defined above. Similarly, two
codes $\mC_1,\mC_2\in\FMR$ are said to be \emph{equivalent} if there exists
$g\in\GMR$ such that $\mC_1\cdot g=\mC_2$. We denote by $[\mC]$ the equivalence
class of a code $\mC$.   
\end{definition}

\subsection{Geometric point of view}\label{subsec:geo_tutto}

For the families $\FLH, \FAH, \FLR$, geometric counterparts are already known. We recall them in this subsection.

\begin{itemize}
    \item[(LH)]  A \emph{projective} $[n,k,d]_\F$ \emph{system} $\sLH:=(\mathcal{P},\mm)$ is a finite multiset of projective points $\mP\subseteq \PP(\F^k)$ not all contained in a hyperplane, together with  a multiplicity function $\mm:\mP\rightarrow \mathbb N_{\geq 0}$, such that $\sum_{P\in\mathcal P}\mm(P)=n$.
    The parameter~$d$ is defined as 
    $$d = n- \max\bigg\{\sum_{P\in \mP\cap \mH}\mm(P) \; : \;  \mH\subseteq \PP(\F^k), \; \dim(\mH) = k-2\bigg\}.$$ 
     Two projective $[n,k,d]_\F$ systems
    $\sLH_1=(\mP_1,\mm_1)$ and $\sLH_2=(\mP_2,\mm_2)$ are
    \emph{equivalent} if there exists a projective isomorphism
    $\varphi\in\PGL_k(\F)$ mapping $\mP_1$ to $\mP_2$ and preserving
    multiplicities, that is, $\mm_1(P)=\mm_2(\varphi(P))$ for every $P\in\PP(\F^k)$. We denote by
    $$\SLH:=\{\text{projective }[n,k,d]_\F\text{ systems} \; : \; d\in [n]\},$$ 
    and by $[\mS]$ the equivalence class of $\mS\in \SLH$. For more details, see \cite{tsfasman2013algebraic}.
    
    \item[(AH)] An \emph{$m$-projective} $[n,k/m,d]_\K^\F$ \emph{system} $\sAH:=(\mathcal{X},\mm)$ is a finite multiset of projective subspaces $\mX\subseteq\PP(\K^k)$ of dimension at most $m-1$,  not all contained in
    a hyperplane, together with a multiplicity function $\mm:\mX\rightarrow \mathbb N_{\geq 0}$,  such that $\sum_{X\in\mathcal \mX}\mm(X)=n$.
    The parameter~$d$ is defined as 
    $$d = n- \max\bigg\{\sum_{X\in\mX, \; X\subseteq \mH}\mm(X) \; : \;  \mH\subseteq \PP(\K^k), \; \dim(\mH) = k-2\bigg\}.$$ 
    If all the subspaces in $\mX$ have dimension $m-1$, we refer to $\mS$ as  \textbf{faithful}.
    Two $m$-projective $[n,k/m,d]_\K^\F$ systems
    $\sAH_1=(\mX_1,\mm_1)$ and $\sAH_2=(\mX_2,\mm_2)$ are
    \emph{equivalent} if there exists a projective isomorphism
    $\varphi\in\PGL_k(\K)$ mapping $\mX_1$ to $\mX_2$ and preserving
    multiplicities, that is, $\mm_1(X)=\mm_2(\varphi(X))$ for every $X\in \mX_1$. We denote by
    $$\SAH:=\{\text{projective }[n,k/m,d]^\F_\K\text{ systems} \; : \; d\in [n]\},$$ 
    and by $[\mS]$ the equivalence class of $\mS\in \SAH$. For more details, see \cite{ball2023additive, kurz2024additive}.

    \item[(LR)] An $[n,k,d]_{\F/\K}$ \emph{system} $\sLR$ is a $\K$-subspace of $\F^k$ of $\K$-dimension $n$, not contained in a hyperplane of $\F^k$. The parameter~$d$ is defined as 
    $$d = n- \max \bigg\{\dim_\K(\sLR\cap\mH) \; : \;  \mH\subseteq \F^k, \; \dim(\mH) = k-1\bigg\}.$$ 
    Projective $[n,k,d]_{\F/\K}$ systems $\sLR_1$ and $\sLR_2$ are \emph{equivalent} if there exists an $\F$-linear isomorphism~$\varphi \in \GL_k(\F)$ mapping $\sLR_1$ to $\sLR_2$. We denote by 
    $$\SLR:=\{[n,k,d]_{\F/\K}\text{ systems} \; : \; d\in [n]\},$$ 
    and by $[\mS]$ the equivalence class of $\mS\in \SLR$. For more details, see \cite{randrianarisoa2020geometric}.
\end{itemize}

\begin{remark}
When the parameters are irrelevant or clear from the context, we simply refer to these objects as \emph{projective systems} in the case (LH),
\emph{$m$-projective systems} in the case (AH), and $\K$-\emph{systems} in the case (LR).
\end{remark}

The following result collects the well-known correspondences between codes and systems in these three settings.

\begin{theorem}\label{thm:correspondences}
For every $\ast\in\{{\rm LH},{\rm AH},{\rm LR}\}$, there is a one-to-one
correspondence between equivalence classes in $\Fast$ and equivalence classes in $\Sast$.
\end{theorem}

\begin{proof}
The correspondence works as follows:
\begin{itemize}
    \item[(LH)] Let $\mC$ be a nondegenerate code in $\FLH$ and $G\in\F^{k\times n}$ be a generator matrix of $\mC$. Let $(\mP,\mm)$ be the multiset of points corresponding to the columns of $G$. Define
    $$\phi^{\rm LH}:\FLH/_\sim\longrightarrow \SLH/_\sim, \ \ \ \ [\mC]\mapsto[(\mP,\mm)].$$
    This map is well defined, bijective, and independent of the choice of the
    generator matrix~$G$. Moreover, for every $u\in \F^k$,
    $$\wt_{\rm H}(uG)=n-\sum_{P\in \mP\cap  u^\perp}\mm(P).$$
    Hence $\phi^{\rm LH}$ sends nondegenerate $[n,k,d]_\F$ codes to
    $[n,k,d]_\F$ systems; see, e.g., \cite{tsfasman2013algebraic}.
    
    \item[(AH)] Let $\mC$ be a nondegenerate code in $\FAH$ and $G\in\F^{k\times n}$ be a generator matrix of $\mC$. Let $(\mX,\mm)$ be the multiset of subspaces of $\PP(\K^k)$ obtained as follows: for each column $g$ of $G$ consider the matrix $\Gamma(g^\top )\in \K^{m\times k}$ defined by expanding the coefficient of $g^\top $ with respect to a basis $\Gamma$ of $\F$ over $\K$, and consider the space $X=\PP(\rowsp_\K(\Gamma(g^\top )))$. Then $\mX$ is the multiset of such subspaces and $\mm$ its multiplicity function.
    Define 
    $$\phi^{\rm AH}:\FAH/_\sim\longrightarrow \SAH/_\sim, \  \ \ \ [\mC]\mapsto[(\mX,\mm)].$$
    This map is well defined, bijective, and independent of the choice of the
    generator matrix~$G$. Moreover, for every $u\in\K^k$,
    $$\wt_{\rm H}(uG)=n-\sum_{X\in\mX,X\subseteq u ^\perp}\mm(X).$$
    Hence $\phi^{\rm AH}$ sends nondegenerate $[n,k,d]_\K^\F$ codes to
    $[n,k,d]_\K^\F$ systems; see, e.g., \cite{ball2023additive,kurz2024additive}.
    
    \item[(LR)] Let $\mC$ be a nondegenerate code in $\FLR$ and $G\in\F^{k\times n}$ be a generator matrix of $\mC$. Let $\mU=\rowsp_\K(G^\top)\subseteq \F^k$ be the $\K$-subspace of $\F^k$ generated by the rows of $G^\top $. Define 
    $$\phi^{\rm LR}:\FLR/_\sim\longrightarrow \SLR/_\sim,  \ \ \ \ [\mC]\mapsto[\mU].$$
    This map is well defined, bijective, and independent of the choice of the
    generator matrix~$G$. Moreover, for every $u\in\F^k$,
    $$\wt_{\rk}(uG)=n-\dim_\K(\mU\cap  u^\perp).$$
     Hence $\phi^{\rm LR}$ sends nondegenerate $[n,k,d]_{\F/\K}$ codes to
    $[n,k,d]_{\F/\K}$ systems; see \cite{alfarano2021linear,randrianarisoa2020geometric}.
\end{itemize}
\end{proof}

Thus, for $\ast\in\{{\rm LH},{\rm AH},{\rm LR}\}$, the family $\Sast$ may be
regarded as the geometric counterpart of $\Fast$. In the next section, we
introduce the corresponding notions for codes in $\FMR$.

\section{The geometry of matrix rank-metric codes}\label{sec:q-system-matrices}

In this section, we fill the gap left by the previous discussion by providing a geometric interpretation for the family $\FMR$. We begin with some notation.

\begin{notation}\label{not:shortening}
    For positive integers $N_1,N_2$, let $\mathcal Z\subseteq \K^{N_1\times N_2}$, $\mU\subseteq \K^{N_1}$, $\mV\subseteq \K^{N_2}$ be $\K$-subspaces. We define
    $$ \mZ_{\ccc}(\mU):=\mZ\cap(\mU\otimes \K^{N_2})=\left\{Z \in \mZ\,:\, \colsp_{\K}(Z)\subseteq \mU\right\},$$
    $$ \mZ_{\rrr}(\mV):=\mZ\cap(\K^{N_1}\otimes \mV)=\left\{Z \in \mZ\,:\, \rowsp_{\K}(Z)\subseteq \mV\right\}.$$
\end{notation}

It is immediate to verify that $\mZ_{\ccc}(\mU)$ and $\mZ_{\rrr}(\mV)$ are $\K$-subspaces of $\K^{N_1\times N_2}$. The spaces of this form are typically used in the context of shortening of rank-metric codes; see for example \cite[Remark~2.9]{alfarano2026} or \cite[Remark~11]{byrne2017covering}. 

We can now define the main object of the paper.

\begin{definition}\label{def:systems}
An $[m\times n,k,d]_\K$ \emph{column-system} is an $n$-dimensional
$\K$-subspace $\sMR$ of $\K^{k\times m}$ such that $\sum_{M\in \sAR}\colsp_\K(M)$ is not contained in a hyperplane of $\K^k$. The parameter $d$ is defined by
$$
d=n-\max\{\dim_\K(\sMR_{\ccc}(\mH))\,:\, \mH\subseteq \K^k, \; \dim_\K(\mH)=k-1\}.
$$ 
When the parameters are clear from the context or not relevant, we simply refer to $\sMR$ as a $\K$-\emph{column-system}.
\end{definition}

We consider the action of $\GL_k(\K)\times\GL_m(\K)$ on $\K^{k\times m}$ in
which $\GL_k(\K)$ acts by multiplication on the left and $\GL_m(\K)$ acts by
multiplication on the right. Two $[m\times n,k,d]_\K$ column-systems
$\sMR_1$ and $\sMR_2$ are said to be \emph{equivalent} if there exists
$(A,B)\in\GL_k(\K)\times\GL_m(\K)$ such that $\sMR_1=A^\top \sMR_2B$. 
We denote by
    $$\SMRc:=\{[m\times n,k,d]_{\K}\text{ column-systems}\; : \; d\in [n]\},$$ 
and by $[\mS]$ the equivalence class of $\mS\in \SMRc$.

\begin{remark}\label{rem:rowsystems}
One can define, in a completely analogous way, the notion of an
$[m\times n,k,d]_\K$ \textbf{row-system}. This is an $m$-dimensional
$\K$-subspace $\mT$ of $\K^{n\times k}$ such that $\sum_{M\in \mT}\rowsp_\K(M)$ is not contained in any hyperplane of $\K^k$, where $d$ is given by
 $$d=m-\max\{\dim_\K(\mT_{\rrr}(\mH))\,:\, \mH\subseteq \K^k, \; \dim_\K(\mH)=k-1\}.$$
 In the same way, we have that two  $[m\times n,k,d]_{\K}$ row-systems $\mT_1$ and $\mT_2$ are \emph{equivalent} if there exists a $(A,B)\in \GL_n(\K)\times \GL_k(\K)$ such that $\mT_1=A^\top \mT_2B$. 
We denote by
    $$\SMRr:=\{[m\times n,k,d]_{\K}\text{ row-systems}\; : \; d\in [m]\},$$ 
and by $[\mT]$ the equivalence class of $\mT\in \SMRr$.
 However, across the paper we will mainly use the notion of column-system, apart from a few special situations.
\end{remark}

\begin{remark}\label{rem:systems_are_codes}
The objects in $\SMRc$ are, in fact, matrix spaces themselves. More precisely,
taking transpose gives a natural bijection between $\SMRc$ and  $\mathscr{F}_{{\rm MR}_{\ccc}}(m,k,n)$. Indeed, if $\mS\subseteq \K^{k\times m}$ is an $n$-dimensional $\K$-column-system, then $\mS^\top$ is an $n$-dimensional $\K$-subspace of
$\K^{m\times k}$. Moreover, the nondegeneracy condition for $\mS$,
namely that $\sum_{M\in\mS}\colsp_\K(M)$ is not contained in a hyperplane of $\K^k$, is precisely the
column-nondegeneracy condition for the matrix code $\mS^\top$.
This identification is compatible with equivalence. Indeed, if
$\mS_1=A^\top\mS_2B$, with $(A,B)\in\GL_k(\K)\times\GL_m(\K)$, then $\mS_1^\top=B^\top \mS_2^\top A$.
Thus equivalent column-systems correspond to equivalent matrix rank-metric
codes. Hence, transposition also induces a bijection between the corresponding
equivalence classes.
\end{remark}

We now give the key ingredient for the main theorem of this section. This is given by the following technical lemma, which generalizes \cite[Theorem 3.1]{neri2023geometry}, and allows to read the row-support of a nonzero codeword in an $[m\times n,k]_\K$ code $\mC$ from a particular $[m\times n,k]_\K$ column-system obtained via a generator tensor $T$ of $\mC$. First of all, observe that any codeword $M$ in $\mC$ is of the form $m_1(T,u)$, for some $u\in\K^k$.

\begin{lemma}\label{lem:supports_and_ranks}
    Let $T$ be a generator tensor of  a column-nondegenerate $[m\times n,k]_\K$ code $\mC$, and let $\mS=\sss_3(T)$. Then, for any $u \in \K^k$, we have
\begin{equation}\label{eq:support}\boxed{\rsupp_\rk(m_1(T,u))^\perp=\psi_{T}^{-1}(\mS_{\ccc}( u^\perp)),}
\end{equation}
where
$$\psi_{T}:\K^n\to \mS, \ \ \  v\mapsto m_3(T,v).$$
In particular,
\begin{equation}\label{eq:rank}
\boxed{\rk(m_1(T,u))=n-\dim_\K\mS_{\ccc}(u^\perp).}    
\end{equation}
\end{lemma}

\begin{proof}
First of all, observe that $\mS=\sss_3(T)$ is an $n$-dimensional $\K$-subspace of $\K^{k\times m}$, since $\mC$ is column-nondegenerate.

If $u=0$, then the statement is clearly true, since $\rsupp_\rk(m_1(T,0))^\perp=0^\perp=\K^{n}$ and $\mS_{\ccc}( 0^\perp)=\mS$. 

Assume $u \neq 0$. For every $u\in\K^k\setminus\{0\}$, the space $u^\perp$ is a hyperplane of $\K^k$ and 
$$\mS_{\ccc}(u^\perp)=\{M\in\mS \; : \; \colsp_\K(M)\subseteq u^\perp\}$$
is a well-defined subspace of $\mS$.
    Let  $r:=n-\dim_\K\mS_{\ccc}(u^\perp)$. Then, we can consider a $\K$-basis of $\mS_{\ccc}(u^\perp)$, say $\{X'_{r+1},\ldots,X'_n\}$ . We can complete it to a $\K$-basis of $\mS$, say $\{X'_1,$ $\ldots,$ $X'_r, \; X'_{r+1}$ $, \ldots, X'_n\}$. Now, consider such matrices as slices of a tensor $T'\in \K^{k\times m\times n}$. There exists $B\in \GL_n(\K)$ such that $T'=m_3(T,B)$. We have that $$m_1(m_3(T,B),u)=m_1(T',u)=[X_1'^\top u^\top | \ldots | X_r'^\top u^\top | X_{r+1}'^\top u^\top | \ldots |X_n'^\top u^\top ]=[Y_1|\ldots|Y_r|0|\ldots|0]$$ with 
$Y_1,\ldots,Y_r\in \K^m$, $\K$-linearly independent.
Consequently, $$\rsupp_\rk(m_1(m_3(T,B),u))=\langle e_1,\ldots,e_r\rangle_\K\subseteq \K^n.$$ Hence, $\rsupp_\rk(m_1(T,u))=\langle e_1,\ldots,e_r\rangle_\K B^{-1}$. 
If we consider the map 
$$\psi_{T}:\K^n\to \mS, \ \ \  v\mapsto m_3(T,v),$$
then we have $\psi_{m_3(T,B)}^{-1}(\mS_{\ccc}(u^\perp))=\langle e_{r+1},\ldots,e_n\rangle_\K\subseteq \K^n$. Finally, we obtain $\psi_{T}^{-1}(\mS_{\ccc}(u^\perp))=\psi_{m_3(T,B)}^{-1}(\mS_{\ccc}(u^\perp))B^\top =\langle e_{r+1},\ldots,e_n\rangle_\K B^\top$, which concludes the proof, since $\langle e_{r+1},\ldots,e_n\rangle_\K B^\top=(\langle e_1,\ldots,e_r\rangle_\K B^{-1})^\perp=\rsupp_\rk(m_1(T,u))^\perp$ .

\end{proof}

The following result is a generalization of \cite[Theorem 2]{randrianarisoa2020geometric} to column-nondegenerate $[m\times n,k,d]_\K$ codes. It provides the geometric framework suitable for studying matrix rank-metric codes.

\begin{theorem}\label{thm:correspondence}
The map 
$$\begin{array}{cccc}\phi^{\rm MR_{\ccc}}:&\FMRc/_\sim&\longrightarrow &\SMRc/_\sim, \\ &[\C]&\longmapsto&[\sss_3(T)],
\end{array}$$
where $T$ is a generator tensor of $\C$, 
is well-defined and provides a one-to-one correspondence between equivalence classes in $\FMRc$ and equivalence classes in $\SMRc$.
\end{theorem}

\begin{proof}
    Let $\mC$ be a code in $\FMRc$ and let $T$ be a generator tensor of $\mC$. Define 
\begin{equation*}\label{eq:pseudosystemassociated}
    \boxed{\sMR:=\sss_3(T).}
\end{equation*}    
     Since $\mC$ is column-nondegenerate, the space $\sMR$ has $\K$-dimension $n$; see Remark \ref{rem:non-degeneracy-space}. 
It is easy to prove that the map $\phi^{\rm MR_{\ccc}}$ is well-defined, bijective, and does not depend on the choice of $T$. Indeed:
\begin{itemize}
    \item If $\mC'\in [\mC]$, then there exists $(A,B)\in \GL_m(\K)\times \GL_n(\K)$ such that $A^\top \mC B:=\{A^\top MB \; : \; M \in \mC\}=\mC'$. If $T$ is a generator tensor for $\mC$, then a generator tensor for $A^\top\mC B$ is $m_2(m_3(T,B),A^\top)$. Clearly $\sss_3(T)=\sss_3(m_3(T,B))$.  The multiplication by $A^\top$ corresponds to an equivalence of $\K$-column-systems. Hence, the map $\phi^{\rm MR_{\ccc}}$ is well-defined.
    \item We show that the map is invertible by constructing the map $$\psi^{\rm MR_{\ccc}}:\SMRc/_\sim\longrightarrow \FMRc/_\sim,$$
    as follows: for a system $\mS\in \SMRc$, choose a basis $X_1,\ldots,X_n\in \K^{k\times m}$. Put these matrices as slices of a tensor $T'\in \K^{k\times m\times n}$ and let $\psi^{\rm MR_{\ccc}}([\mS]):=[\mC']$, where $\mC':=\sss_1(T')$. One can show that this map is well-defined, that $\psi^{\rm MR_{\ccc}}(\phi^{\rm MR_{\ccc}}([\mC]))=[\mC]$ for every $\C\in \FMRc$ and that $\phi^{\rm MR_{\ccc}}(\psi^{\rm MR_{\ccc}}([\mS]))=[\mS]$ for every $\mS\in \SMRc$. 
    \item If $T'$ is another generator tensor of $\mC$, there exists $A\in \GL_k(\K)$ such that $T'=m_1(T,A)$. Then, the $\K$-column-system associated with $T'$ is $\sMR'=A\sMR$, which is equivalent to $\sMR$.
\end{itemize}
Finally, the preservation of the minimum distance follows from Lemma
\ref{lem:supports_and_ranks}. Indeed, for every $u\in\K^k$, $\rk(m_1(T,u))
    =
    n-\dim_\K \sMR_{\ccc}(u^\perp)$.
Thus $\phi^{{\rm MR}_{\ccc}}$ sends equivalence classes of column-nondegenerate
$[m\times n,k,d]_\K$ codes to equivalence classes of
$[m\times n,k,d]_\K$ column-systems, completing the proof.
\end{proof}

\begin{definition}
We say that a $\K$-column-system $\mS$ is \textbf{associated with} a code
$\mC\in\FMRc$ if $\phi^{{\rm MR}_{\ccc}}([\mC])=[\mS]$. In this case, we also say that $\mC$ is a \textbf{code associated with} the
$\K$-column-system~$\mS$.
\end{definition}

\begin{example}\label{ex:toy}
Consider the field $\K = \F_2$. Let $\C$ be the $[3\times 3,4,2]_2$ code generated by the matrices
$$B_1=\begin{pmatrix}
\textcolor{red}{1}&\textcolor{ForestGreen}{0}&\textcolor{blue}{0}\\
\textcolor{red}{0}&\textcolor{ForestGreen}{0}&\textcolor{blue}{0}\\
\textcolor{red}{1}&\textcolor{ForestGreen}{1}&\textcolor{blue}{0}
\end{pmatrix}, \ 
B_2=\begin{pmatrix}
\textcolor{red}{0}&\textcolor{ForestGreen}{0}&\textcolor{blue}{1}\\
\textcolor{red}{0}&\textcolor{ForestGreen}{0}&\textcolor{blue}{0}\\
\textcolor{red}{1}&\textcolor{ForestGreen}{1}&\textcolor{blue}{1}
\end{pmatrix}, \ 
B_3=\begin{pmatrix}
\textcolor{red}{0}&\textcolor{ForestGreen}{0}&\textcolor{blue}{0}\\
\textcolor{red}{1}&\textcolor{ForestGreen}{0}&\textcolor{blue}{0}\\
\textcolor{red}{0}&\textcolor{ForestGreen}{1}&\textcolor{blue}{1}
\end{pmatrix}, \
B_4=\begin{pmatrix}
\textcolor{red}{0}&\textcolor{ForestGreen}{0}&\textcolor{blue}{0}\\
\textcolor{red}{0}&\textcolor{ForestGreen}{1}&\textcolor{blue}{0}\\
\textcolor{red}{1}&\textcolor{ForestGreen}{1}&\textcolor{blue}{1}
\end{pmatrix}.$$

We regard these matrices as the horizontal slices, more precisely the
$1$-slices, of a tensor $T\in\F_2^{4\times 3\times 3}$. The map
$\phi^{{\rm MR}_{\ccc}}$ then associates with $\mC$ the column-system obtained
by taking the vertical slices of $T$, as illustrated below.

\medskip

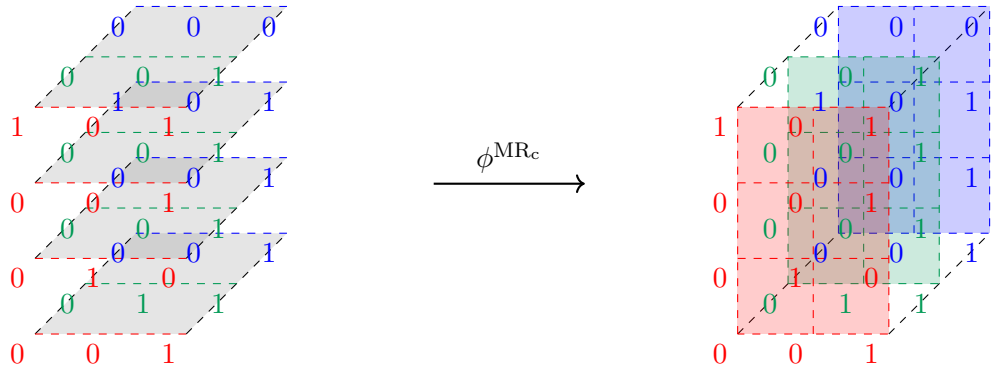
\begin{minipage}[t]{0.45\textwidth}
\centering
\begin{tikzpicture}
\fill[gray, fill opacity=0.2] (0,0) -- (2,0) -- (2+4/3,4/3) -- (4/3,4/3) -- (0,0);
\fill[gray, fill opacity=0.2] (0,+1) -- (2,0+1) -- (2+4/3,4/3+1) -- (4/3,4/3+1) -- (0,0+1);
\fill[gray, fill opacity=0.2] (0,0+2) -- (2,0+2) -- (2+4/3,4/3+2) -- (4/3,4/3+2) -- (0,0+2);
\fill[gray, fill opacity=0.2] (0,0+3) -- (2,0+3) -- (2+4/3,4/3+3) -- (4/3,4/3+3) -- (0,0+3);

\draw[dashed] (0,0) -- (4/3,4/3);
\draw[dashed] (0,1) -- (4/3,4/3+1);
\draw[dashed] (0,2) -- (4/3,4/3+2);
\draw[dashed] (0,3) -- (4/3,3+4/3);
\draw[dashed] (2,0) -- (2+4/3,4/3);
\draw[dashed] (2,1) -- (2+4/3,4/3+1);
\draw[dashed] (2,2) -- (2+4/3,4/3+2);
\draw[dashed] (2,3) -- (2+4/3,3+4/3);


\draw[dashed,blue] (0+4/3,0+4/3) -- (2+4/3,0+4/3);
\draw[dashed,blue] (0+4/3,3+4/3) -- (2+4/3,3+4/3);
\draw[dashed,blue] (0+4/3,1+4/3) -- (2+4/3,1+4/3);
\draw[dashed,blue] (0+4/3,2+4/3) -- (2+4/3,2+4/3);

\draw[dashed,ForestGreen] (0+2/3,0+2/3) -- (2+2/3,0+2/3);
\draw[dashed,ForestGreen] (0+2/3,3+2/3) -- (2+2/3,3+2/3);
\draw[dashed,ForestGreen] (0+2/3,1+2/3) -- (2+2/3,1+2/3);
\draw[dashed,ForestGreen] (0+2/3,2+2/3) -- (2+2/3,2+2/3);

\draw[dashed,red] (0,0) -- (2,0);
\draw[dashed,red] (0,3) -- (2,3);
\draw[dashed,red] (0,1) -- (2,1);
\draw[dashed,red] (0,2) -- (2,2);

\foreach \x/\y in {0/0,0/1,0/2,1/0,1/2,1/3,2/1}
{
    \node[anchor=north east,red] at (\x,\y) {0};
}
\foreach \x/\y in {0/3,1/1,2/0,2/2,2/3}
{
    \node[anchor=north east,red] at (\x,\y) {1};
}
\foreach \x/\y in {0.66/0.66,0.66/1.66,0.66/2.66,0.66/3.66,1.66/1.66,1.66/2.66,1.66/3.66}
{
    \node[anchor=north east,ForestGreen] at (\x,\y) {0};
}
\foreach \x/\y in {1.66/0.66,2.66/0.66,2.66/1.66,2.66/2.66,2.66/3.66}
{
    \node[anchor=north east,ForestGreen] at (\x,\y) {1};
}

\foreach \x/\y in {1.33/1.33,1.33/2.33,1.33/4.33,2.33/1.33,2.33/2.33,2.33/3.33,2.33/4.33,3.33/4.33}
{
    \node[anchor=north east,blue] at (\x,\y) {0};
}
\foreach \x/\y in {1.33/3.33,3.33/1.33,3.33/2.33,3.33/3.33}
{
    \node[anchor=north east,blue] at (\x,\y) {1};
}
\end{tikzpicture}
\end{minipage}
\begin{minipage}[t]{0.1\textwidth}
\begin{tikzpicture}[baseline]
  \draw[->, thick] (0,2.5) -- (2,2.5) node[midway, above] {$\phi^{\rm MR_{\ccc}}$};
\end{tikzpicture}
\end{minipage}
\begin{minipage}[t]{0.45\textwidth}
\centering
\begin{tikzpicture}

\draw[dashed] (0,0) -- (4/3,4/3);
\draw[dashed] (0,3) -- (4/3,3+4/3);
\draw[dashed] (2,0) -- (2+4/3,4/3);
\draw[dashed] (2,3) -- (2+4/3,3+4/3);

\draw[dashed,blue] (1+4/3,0+4/3) -- (1+4/3,3+4/3);
\draw[dashed,blue] (0+4/3,1+4/3) -- (2+4/3,1+4/3);
\draw[dashed,blue] (0+4/3,2+4/3) -- (2+4/3,2+4/3);
\draw[dashed,blue] (4/3,4/3) rectangle (2+4/3,3+4/3);
\fill[blue, fill opacity=0.2] (4/3,4/3) rectangle (2+4/3,3+4/3);

\draw[dashed,ForestGreen] (1+2/3,0+2/3) -- (1+2/3,3+2/3);
\draw[dashed,ForestGreen] (0+2/3,1+2/3) -- (2+2/3,1+2/3);
\draw[dashed,ForestGreen] (0+2/3,2+2/3) -- (2+2/3,2+2/3);
\draw[dashed,ForestGreen] (2/3,2/3) rectangle (2+2/3,3+2/3);
\fill[ForestGreen, fill opacity=0.2] (2/3,2/3) rectangle (2+2/3,3+2/3);

\draw[dashed,red] (1,0) -- (1,3);
\draw[dashed,red] (0,1) -- (2,1);
\draw[dashed,red] (0,2) -- (2,2);
\draw[dashed,red] (0,0) rectangle (2,3);
\fill[red, fill opacity=0.2] (0,0) rectangle (2,3);

\foreach \x/\y in {0/0,0/1,0/2,1/0,1/2,1/3,2/1}
{
    \node[anchor=north east,red] at (\x,\y) {0};
}
\foreach \x/\y in {0/3,1/1,2/0,2/2,2/3}
{
    \node[anchor=north east,red] at (\x,\y) {1};
}
\foreach \x/\y in {0.66/0.66,0.66/1.66,0.66/2.66,0.66/3.66,1.66/1.66,1.66/2.66,1.66/3.66}
{
    \node[anchor=north east,ForestGreen] at (\x,\y) {0};
}
\foreach \x/\y in {1.66/0.66,2.66/0.66,2.66/1.66,2.66/2.66,2.66/3.66}
{
    \node[anchor=north east,ForestGreen] at (\x,\y) {1};
}

\foreach \x/\y in {1.33/1.33,1.33/2.33,1.33/4.33,2.33/1.33,2.33/2.33,2.33/3.33,2.33/4.33,3.33/4.33}
{
    \node[anchor=north east,blue] at (\x,\y) {0};
}
\foreach \x/\y in {1.33/3.33,3.33/1.33,3.33/2.33,3.33/3.33}
{
    \node[anchor=north east,blue] at (\x,\y) {1};
}
\end{tikzpicture}
\end{minipage}

\medskip

This gives the three matrices
$$X_1=\begin{pmatrix}
\textcolor{red}{1}&\textcolor{red}{0}&\textcolor{red}{1}\\
\textcolor{red}{0}&\textcolor{red}{0}&\textcolor{red}{1}\\
\textcolor{red}{0}&\textcolor{red}{1}&\textcolor{red}{0}\\
\textcolor{red}{0}&\textcolor{red}{0}&\textcolor{red}{1}
\end{pmatrix}, \; 
X_2=\begin{pmatrix}
\textcolor{ForestGreen}{0}&\textcolor{ForestGreen}{0}&\textcolor{ForestGreen}{1}\\
\textcolor{ForestGreen}{0}&\textcolor{ForestGreen}{0}&\textcolor{ForestGreen}{1}\\
\textcolor{ForestGreen}{0}&\textcolor{ForestGreen}{0}&\textcolor{ForestGreen}{1}\\
\textcolor{ForestGreen}{0}&\textcolor{ForestGreen}{1}&\textcolor{ForestGreen}{1}
\end{pmatrix}, \;
X_3=\begin{pmatrix}
\textcolor{blue}{0}&\textcolor{blue}{0}&\textcolor{blue}{0}\\
\textcolor{blue}{1}&\textcolor{blue}{0}&\textcolor{blue}{1}\\
\textcolor{blue}{0}&\textcolor{blue}{0}&\textcolor{blue}{1}\\
\textcolor{blue}{0}&\textcolor{blue}{0}&\textcolor{blue}{1}
\end{pmatrix}.
$$
Let $\mS=\langle X_1,X_2,X_3\rangle_{\F_2}$. We now illustrate Eq.
\eqref{eq:rank}. Consider the vector $u=(1,0,1,0)\in\F_2^4$. The corresponding
codeword is $$c=m_1(T,u)=B_1+B_3=\begin{pmatrix}
\textcolor{red}{1}&\textcolor{ForestGreen}{0}&\textcolor{blue}{0}\\
\textcolor{red}{1}&\textcolor{ForestGreen}{0}&\textcolor{blue}{0}\\
\textcolor{red}{1}&\textcolor{ForestGreen}{0}&\textcolor{blue}{1}
\end{pmatrix},$$
which has rank $2$. Eq. \eqref{eq:rank} gives the same value as $\rk(c)
=
3-\dim_{\F_2}(\mS_{\ccc}(u^\perp))$. Indeed, $u^\perp=\{(a,b,a,c):a,b,c\in \F_2\}$ and the only nonzero matrix in $\mS$ whose column space is contained in
$u^\perp$ is $X_2$. Hence, $\dim_{\F_2}(\mS_{\ccc}(u^\perp))=1$, and therefore $3-\dim_{\F_2}(\mS_{\ccc}(u^\perp))=3-1=2$.
\end{example}

\begin{remark}
    Analogues of Lemma \ref{lem:supports_and_ranks} and
Theorem \ref{thm:correspondence} hold for row-systems as well, as defined in Remark \ref{rem:rowsystems}. More precisely,
let $T$ be a generator tensor of an $[m\times n,k]_\K$ code $\mC$, and set $\mT:=\sss_2(T)^\top$.
Then, $\mT$ is the row-system naturally associated with $\mC$ via $T$. In this
case, for every $u\in\K^k$, we have
    \begin{equation}\label{eq:col-support}\boxed{\csupp_\rk(m_1(T,u))^\perp={\varphi}_{T}^{-1}(\mT_{\rrr}(u^\perp)),}
\end{equation}
where $$\varphi_{T}:\K^m\to \mT, \ \ \  v\mapsto m_2(T,v)^\top.$$

Moreover, if we consider the set 
   $$\SMRr:=\{[m\times n,k,d]_{\K}\text{ row-systems}\; : \; d\in [m]\},$$ 
then there is a well-defined one-to-one correspondence 
$$\begin{array}{cccc}\phi^{\rm MR_{\rrr}}:&\FMRr/_\sim&\longrightarrow &\SMRr/_\sim, \\ &[\C]&\longmapsto&[\sss_2(T)^\top].
\end{array}$$
where $T\in\K^{k\times m\times n}$ is a generator tensor for $\mC$, that is, $\mC=\sss_1(T)$. Here, the equivalence relation on $\SMRr$ is induced by the natural action of $\GL_n(\K)\times\GL_k(\K)$.
Also in this case, we say that a $\K$-row-system $\mT$ is \emph{associated
with} a code $\mC$ if $\phi^{\rm MR_{\rrr}}([\mC])=[\mT]$.
\end{remark}

\begin{remark}
    Let $\C$ be a column- and row-nondegenerate $[m\times n,k]_\K$ code. Then one can apply to~$[\C]$ both the map $\phi^{\rm MR_{\ccc}}$ and $\phi^{\rm MR_{\rrr}}$. In the first case, by Remark~\ref{rem:systems_are_codes}, the class $\phi^{\rm MR_{\ccc}}([\C])$ can be identified with an equivalence class of column- and row-nondegenerate $[k\times m,n]_\K$ codes. In the second case, the class $\phi^{\rm MR_{\rrr}}([\C])$ can be identified with an equivalence class of column- and row-nondegenerate $[n\times k,m]_\K$ codes. In particular, if we denote by
    $$\FMR:=\FMRc\cap\FMRr$$
    the family of column- and row-nondegenerate $[m\times n,k]_\K$ codes, then, up to the identifications of Remark~\ref{rem:systems_are_codes}, we have maps
    $$\phi^{\rm MR_{\ccc}}:\FMR/\sim \;\;\longrightarrow\; \mathscr{F}_{\rm MR}(k,m,n)/\sim,$$
    $$\phi^{\rm MR_{\rrr}}:\FMR/\sim\;\; \longrightarrow\; \mathscr{F}_{\rm MR}(n,k,m)/\sim.$$
    By abuse of notation, one can compose these maps. Indeed, if $\phi^{\rm MR_{\ccc}}([\C])$ is identified with an equivalence class of $[k\times m,n]_\K$ codes, then one can again consider the associated column- or row-system of any representative. The same applies to $\phi^{\rm MR_{\rrr}}$. At the level of generator tensors, these two operations correspond to the two cyclic permutations of the three tensor directions. Hence they satisfy the relations:
    $$\begin{cases}
        \phi^{\rm MR_{\ccc}}\circ \phi^{\rm MR_{\ccc}} \circ \phi^{\rm MR_{\ccc}}=\mathrm{id}, \\
        \phi^{\rm MR_{\rrr}}\circ \phi^{\rm MR_{\rrr}} \circ \phi^{\rm MR_{\rrr}}=\mathrm{id},\\
        \phi^{\rm MR_{\ccc}}\circ \phi^{\rm MR_{\rrr}} =\mathrm{id},\\
        \phi^{\rm MR_{\rrr}}\circ \phi^{\rm MR_{\ccc}} =\mathrm{id}.
    \end{cases}$$
In other words, on the space of equivalence classes of column- and
row-nondegenerate matrix rank-metric codes, the operation $\phi^{\rm MR_{\ccc}}$ is the inverse of $\phi^{\rm MR_{\rrr}}$,  and
both have order three.

\end{remark}

Over finite fields, the correspondence established in Theorem \ref{thm:correspondence} yields an important relation between the rank distribution of an $[m\times n,k]_q$ code and the rank distribution of any associated $[m\times n,k]_q$ column-system. By Remark \ref{rem:systems_are_codes}, the latter may itself be viewed as a $[k\times m,n]_q$ matrix rank-metric code.

\begin{theorem}\label{thm:StandardEq} 
    Let $\mC$ be a column-nondegenerate $[m\times n,k,d]_q$  code and let $\mS$ be any of its associated  $[m\times n,k,d]_q$ column-systems. Then
    $$\sum_{M\in \mC^*}(q^{n-\rk(M)}-1)=\sum_{X\in\mS^*}(q^{k-\rk(X)}-1).$$
\end{theorem}

\begin{proof}
Since equivalent column-systems have the same rank distribution, it is enough to prove the identity for a representative of the associated
equivalence class. Thus, choose a generator tensor $T$ for $\mC$ and let $\mS$ be the column-system associated with $\mC$ via $T$, that is $\mS=\sss_3(T)$. Let $u \in \Fq^k\setminus\{0\}$, and let $M=m_1(T,u)$. By Lemma \ref{lem:supports_and_ranks}, we have 
    $$ \dim_\K(\mS_{\ccc}(u^\perp))=n-\rk(M),$$
    and, switching to the cardinality, we deduce
    $$|\mS_{\ccc}(u^\perp)\setminus\{0\}|=q^{n-\rk(M)}-1.$$
    If we sum over all the nonzero vectors of $\Fq^k$, we obtain
    $$ \sum_{u \in \Fq^k\setminus\{0\}}|\mS_{\ccc}(u^\perp)\setminus\{0\}|=\sum_{M \in \mC^*}(q^{n-\rk(M)}-1).$$
      On the other hand, 
   \begin{align*} \sum_{u \in \Fq^k\setminus\{0\}}|\mS_{\ccc}(u^\perp)\setminus\{0\}|&=\sum_{u\in\Fq^k\setminus\{0\}}|\{X \in \mS\setminus\{0\}\,:\, \colsp_{\F_q}(X)\subseteq  u^\perp\}|\\
   &=\sum_{X\in\mS\setminus\{0\}}|\{u \in \Fq^k\setminus\{0\}\,:\, \colsp_{\F_q}(X)\subseteq  u^\perp\}|\\
   &=\sum_{X\in\mS\setminus\{0\}}|\{u \in \Fq^k\setminus\{0\}\,:\, u \in \colsp_{\F_q}(X)^\perp\}|\\
   &=\sum_{X\in\mS\setminus\{0\}}(q^{k-\rk(X)}-1),
   \end{align*}
   where the second equality follows from a double-counting argument.
\end{proof}

\begin{example}
Let $\mC$ and $\mS$ be defined as in Example \ref{ex:toy}. The code $\mC$ contains $11$ matrices of rank $2$ and $4$ matrices of rank $3$, whereas  $\mS$, seen as a rank-metric code, contains $2$ matrices of rank $2$ and $5$ matrices of rank $3$. One can then verify that Theorem \ref{thm:StandardEq} holds:
$$\sum_{M\in \mC^*}(2^{3-\rk(M)}-1)=11\cdot(2^{3-2}-1)+4\cdot (2^{3-3}-1)=11=2\cdot (2^{4-2}-1)+5\cdot (2^{4-3}-1)=\sum_{X\in\mS^*}(q^{4-\rk(X)}-1).$$
Furthermore, we can consider the $[3\times 3,4,2]_2$ row-system $\mT$ obtained by $\sss_2(T)^\top$, where $T\in \F_2^{4\times 3\times 3}$ is the tensor given in Example \ref{ex:toy}. The row-system $\mT$ is generated by the following $3$ matrices:
$$ Y_1=\begin{pmatrix}
    \textcolor{red}{1} & \textcolor{red}{0} & \textcolor{red}{0} & \textcolor{red}{0} \\
    \textcolor{ForestGreen}{0} &     \textcolor{ForestGreen}{0} &     \textcolor{ForestGreen}{0} &     \textcolor{ForestGreen}{0} \\
    \textcolor{blue}{0} &     \textcolor{blue}{1} &     \textcolor{blue}{0} &     \textcolor{blue}{0} 
\end{pmatrix}, \ Y_2=\begin{pmatrix}
    \textcolor{red}{0} & \textcolor{red}{0} & \textcolor{red}{1} & \textcolor{red}{0} \\
    \textcolor{ForestGreen}{0} &     \textcolor{ForestGreen}{0} &     \textcolor{ForestGreen}{0} &     \textcolor{ForestGreen}{1} \\
    \textcolor{blue}{0} &     \textcolor{blue}{0} &     \textcolor{blue}{0} &     \textcolor{blue}{0} 
\end{pmatrix}, \ Y_3=\begin{pmatrix}
    \textcolor{red}{1} & \textcolor{red}{1} & \textcolor{red}{0} & \textcolor{red}{1} \\
    \textcolor{ForestGreen}{1} &     \textcolor{ForestGreen}{1} &     \textcolor{ForestGreen}{1} &     \textcolor{ForestGreen}{1} \\
    \textcolor{blue}{0} &     \textcolor{blue}{1} &     \textcolor{blue}{1} &     \textcolor{blue}{1} 
\end{pmatrix}. $$
 We can also verify that the same identity holds for the row-system, as presented in Remark \ref{rem:StEq}. Indeed, in  $\mT$ there are $2$ matrices of rank $2$ and $5$ matrices of rank $3$, providing the same equality as above.
\end{example}

\begin{remark}
Theorem~\ref{thm:StandardEq} may be interpreted as a first
Delsarte-type identity for matrix rank-metric codes. Indeed, its proof
is a double-counting argument for the incidences between the nonzero
elements of an associated column-system and the hyperplanes of
$\Fq^k$. More precisely, it counts pairs
\[
(u,X)\in (\Fq^k\setminus\{0\})\times \mS^*
\quad \mbox{ such that } \quad 
\colsp_{\Fq}(X)\subseteq u^\perp .
\]
Thus, in analogy with the standard equations for projective systems,
Theorem~\ref{thm:StandardEq} should be regarded as a first
incidence identity.
\end{remark}

\begin{definition}
Let $\C$ be an $[m\times n,k]_q$ matrix rank-metric code. For
$0\le t\le k$ and $0\le r\le n$, we define
\[
A_r^{(t)}(\C)
:=|\left\{
\mD\subseteq \C \, :\,
\dim_{\Fq}(\mD)=t,\ 
\dim_{\Fq}(\rsupp_{\rk}(\mD))=r
\right\}|.
\]
We call the sequence
\[
\left(A_0^{(t)}(\C),A_1^{(t)}(\C),\ldots,A_n^{(t)}(\C)\right)
\]
the \textbf{\(t\)-th row-support distribution of \(\C\)}. When $t=1$, the first row-support distribution is strongly related to the rank weight distribution of the code $\C$, namely $(A_0(\C),\ldots,A_n(\C))$ where
$$A_r(\C):=|\{M \in \C \ : \ \rk(M)=r\}|,$$
via the equality $A_r(\C)=(q-1)A_r^{(1)}(\C)$.
\end{definition}

Theorem \ref{thm:StandardEq} admits a higher-dimensional version obtained by replacing nonzero codewords with subcodes of fixed dimension. This will also provide a bridge with the generalized weights studied in Section \ref{sec:bounds}. Before providing such more general result, we start by giving a higher-dimensional version of Lemma \ref{lem:supports_and_ranks}.

\begin{lemma}
\label{lem:higher-support-column-system}
Let $T$ be a generator tensor of a column-nondegenerate
$[m\times n,k]_\K$ code $\C$, and let $\mS=\sss_3(T)$.
For every subspace $\mU\subseteq \K^k$, set
\[
\mD_\mU:=\{m_1(T,u) \, :\, u\in \mU\}\subseteq \C.
\]
Then
\[
\rsupp_{\rk}(\mD_\mU)^\perp
=
\psi_T^{-1}\bigl(\mS_c(\mU^\perp)\bigr),
\]
where
\[
\psi_T:\K^n\longrightarrow \mS,\ \ \
v\longmapsto m_3(T,v).
\]
In particular,
\[
\dim_{\K}(\mS_c(\mU^\perp))
=
n-\dim_{\K}(\rsupp_{\rk}(\mD_\mU)).
\]
\end{lemma}

\begin{proof}
The proof is the same as that of Lemma~\ref{lem:supports_and_ranks},
with the one-dimensional space $\langle u\rangle_{\K}$
replaced by the subspace $\mU$. Indeed, an element
$\psi_T(v)=m_3(T,v)\in \mS$ belongs to $\mS_c(\mU^\perp)$ if and only if
\[
\colsp_{\K}(m_3(T,v))\subseteq \mU^\perp .
\]
This is equivalent to saying that, for every $u\in \mU$,
\[
m_1(T,u)v^\top=0.
\]
Equivalently, $v$ is orthogonal to the row support of every element of $\mD_\mU$. Hence $v\in \rsupp_{\rk}(\mD_\mU)^\perp$.
This proves the equality after applying $\psi_T^{-1}$. The dimension
formula follows because $\psi_T$ is an isomorphism.
\end{proof}

We are now ready to propose the following higher-dimensional version of Theorem \ref{thm:StandardEq}, which can be seen as a set of Delsarte-type identities in the framework of matrix rank-metric codes. 

\begin{theorem}[Higher-order Delsarte-type identities]
\label{thm:higher-delsarte-identities}
Let $\C$ be a column-nondegenerate $[m\times n,k]_q$ code, and let
$\mS$ be any associated $[m\times n,k]_q$ column-system. For every
$t\in\{0,\ldots,k\}$, we have
\[
\sum_{r=0}^{n}
A_r^{(t)}(\C)\bigl(q^{n-r}-1\bigr)
=
\sum_{X\in \mS^*}
\binom{k-\rk(X)}{t}_q .
\]
Equivalently, if $(A_i(\mS))_i$ denotes the rank distribution of $\mS$,
then
\[
\sum_{r=0}^{n}
A_r^{(t)}(\C)\bigl(q^{n-r}-1\bigr)
=
\sum_{i=1}^{\min\{k,m\}}
A_i(\mS)\binom{k-i}{t}_q .
\]
\end{theorem}

\begin{proof}
Since the identity is invariant under equivalence of column-systems, it is enough to prove it for a representative of the associated class. Choose a
generator tensor $T\in\Fq^{k\times m\times n}$ for
$\C$, and let $\mS=\sss_3(T)$. Since $\C$ is
column-nondegenerate, the map
\[
\psi_T:\Fq^n\longrightarrow \mS,\ \ \
v\longmapsto m_3(T,v)
\]
is an isomorphism.

For every $t$-dimensional subspace $\mU\subseteq \Fq^k$, let
\[
\mD_\mU:=\{m_1(T,u) \, :\, u\in \mU\}\subseteq \mC.
\]
Since $T$ is a generator tensor of a $k$-dimensional code, the map
$u\mapsto m_1(T,u)$ is injective, and therefore
\[
\dim_{\Fq}(\mD_\mU)=t.
\]
Moreover, by  Lemma~\ref{lem:higher-support-column-system},
we have
\[
\rsupp_{\rk}(\mD_\mU)^\perp
=
\psi_T^{-1}\bigl(\mS_c(\mU^\perp)\bigr).
\]
Hence
\[
\dim_{\Fq}(\mS_c(\mU^\perp))
=
n-\dim_{\Fq}(\rsupp_{\rk}(\mD_\mU)).
\]

We now double-count the set
\[
\mathcal I_t
:=
\left\{
(\mU,X) \, :\,
\mU\subseteq \Fq^k,\ \dim_{\Fq}(\mU)=t,\ 
X\in \mS^*,\
\colsp_{\Fq}(X)\subseteq \mU^\perp
\right\}.
\]
Counting first with respect to $\mU$, we obtain
\[
|\mathcal I_t|=
\sum_{\substack{\mU\subseteq \Fq^k\\ \dim_{\Fq}(\mU)=t}}
\left(|\mS_c(\mU^\perp)|-1\right)
=
\sum_{\substack{\mU\subseteq \Fq^k\\ \dim_{\Fq}(\mU)=t}}
\left(q^{n-\dim(\rsupp_{\rk}(\mD_\mU))}-1\right).
\]
Grouping the terms according to the value of
$\dim_{\Fq}(\rsupp_{\rk}(\mD_\mU))$, this becomes
\[
|\mathcal I_t|=
\sum_{r=0}^{n}
A_r^{(t)}(\C)\bigl(q^{n-r}-1\bigr).
\]

On the other hand, counting first with respect to $X\in \mS^*$, the
condition
\[
\colsp_{\Fq}(X)\subseteq \mU^\perp
\]
is equivalent to
\[
\mU\subseteq \colsp_{\Fq}(X)^\perp .
\]
If \(\rk(X)=i\), then
\[
\dim_{\Fq}(\colsp_{\Fq}(X)^\perp)=k-i.
\]
Therefore, the number of $t$-dimensional subspaces $\mU\subseteq\Fq^k$
satisfying the above condition is
\[
\binom{k-i}{t}_q
=
\binom{k-\rk(X)}{t}_q .
\]
Thus
\[
|\mathcal I_t|=
\sum_{X\in \mS^*}
\binom{k-\rk(X)}{t}_q ,
\]
which proves the first identity. The second identity follows immediately
by grouping the nonzero elements of $\mS$ according to their rank.
\end{proof}

\begin{remark}
For $t=1$, Theorem~\ref{thm:higher-delsarte-identities} recovers
Theorem~\ref{thm:StandardEq}. Indeed, since
$(q-1)A_r^{(1)}(\C)={A_r(\C)}$ and  $\binom{k-i}{1}_q=\frac{q^{k-i}-1}{q-1}$,
then multiplying the identity of
Theorem~\ref{thm:higher-delsarte-identities} for $t=1$ by $q-1$
gives
\[
\sum_{M\in \C^*}\bigl(q^{n-\rk(M)}-1\bigr)
=
\sum_{X\in \mS^*}\bigl(q^{k-\rk(X)}-1\bigr),
\]
which is precisely Theorem~\ref{thm:StandardEq}. Note that, in terms of rank weight distributions, this can be rewritten as
    $$\sum_{i=1}^{\min\{m,n\}}A_i(\C)(q^{n-i}-1)=\sum_{i=1}^{\min\{k,m\}}A_i(\mS)(q^{k-i}-1).$$
\end{remark}

\begin{remark}\label{rem:StEq}
A completely analogous statement holds for row-systems. If $\mC$ is
row-nondegenerate and $\mT$ is an associated row-system, define
\[
B_r^{(t)}(\C):=
|\left\{
\mD\subseteq \C \, :\,
\dim_{\Fq}(\mD)=t,\ 
\dim_{\Fq}(\csupp_{\rk}(\mD))=r
\right\}|.
\]
Then, for every $t\in\{0,\ldots,k\}$,
\[
\sum_{r=0}^{m}
B_r^{(t)}(\C)\bigl(q^{m-r}-1\bigr)
=
\sum_{Y\in \mT^*}
\binom{k-\rk(Y)}{t}_q .
\]
Equivalently, in terms of the rank weight distribution of \(\mT\),
\[
\sum_{r=0}^{m}
B_r^{(t)}(\C)\bigl(q^{m-r}-1\bigr)
=
\sum_{i=1}^{\min\{k,n\}}
A_i(\mT)\binom{k-i}{t}_q.
\]
Moreover, for $t=1$, we obtain 
$$\sum_{M\in \mC^*}(q^{m-\rk(M)}-1)=\sum_{Y\in\mT^*}(q^{k-\rk(Y)}-1),$$
or, equivalently,
    $$\sum_{i=1}^{\min\{m,n\}}A_i(\C)(q^{m-i}-1)=\sum_{i=1}^{\min\{n,k\}}A_i(\mT)(q^{k-i}-1).$$
\end{remark}

\subsection{Link with $\K$-systems}

As we have already seen, if $[\F:\K]=m$ and $\Gamma=\{\gamma_1,\ldots,\gamma_m\}$ is a basis of $\F$ over $\K$, we can associate to each vector $v\in \F^n$ a matrix $\Gamma(v)\in \K^{m\times n}$ by expanding the coefficients of $v$ with respect to the basis $\Gamma$. This yields a map 
$$\begin{array}{rrcl}
\Gamma:&\FLR &\longrightarrow &\mathscr{F}_{\rm MR_{\ccc}}(m,n,km), \\ 
&\mC&\longmapsto&\{\Gamma(c)\, :\, c \in \mC\}.\end{array}$$
The map is well-defined on equivalence classes. 
We can associate an equivalence class of $[m\times n,km,d]_\K$ column-system to an equivalence class of $[n,k,d]_{\F/\K}$ systems by considering the composition $\phi^{\rm MR_{\ccc}}\circ \Gamma\circ(\phi^{\rm LR})^{-1}$. 

For the sake of completeness, we describe a direct map $\nu_\Gamma$ from $[n,k,d]_{\F/\K}$ systems to $[m\times n,km,d]_\K$. Recall that an $[n,k,d]_{\F/\K}$ system is a $\K$-subspace $\mS$ of $\F^k$ of $\K$-dimension $n$, not contained in a hyperplane of $\F^k$. Let us define the map:
$$
\begin{array}{rrcl}
\nu_\Gamma:&\F^k& \to& \K^{km\times m}\\ 
& (v_1,\ldots,v_k) & \mapsto &(\mu_\Gamma(m_{v_1})^\top|\cdots|\mu_\Gamma(m_{v_k})^\top)^\top,
\end{array}$$
where $\mu_\Gamma(m_{v_i})$ is the representative matrix of the multiplication by $v_i$ with respect to the basis~$\Gamma$. Note that here we are considering row vectors and multiplication on the right by the matrix $\mu_\Gamma(m_{v_1})$. We illustrate this process in the following example.

\begin{example}
    Let $\K = \F_2$ and $\F = \F_4 = \F_2(\alpha)$, where $\alpha^2 + \alpha + 1 = 0$. We choose the basis $\Gamma = \{1, \alpha\}$. The degree of the extension is $m=2$. For each $a \in \F_4$, the matrix $\mu_\Gamma(m_a) \in \K^{2 \times 2}$ is described as follows. $$\mu_\Gamma(m_1) = \begin{pmatrix} 1 & 0 \\ 0 & 1 \end{pmatrix}, \qquad \mu_\Gamma(m_\alpha) = \begin{pmatrix} 0 & 1 \\ 1 & 1 \end{pmatrix}, \qquad \mu_\Gamma(m_{\alpha^2}) = \begin{pmatrix} 1 & 1 \\ 1 & 0 \end{pmatrix}.$$
    Let $k=2$. Consider the vector $v = (1, \alpha) \in \F_4^2$. The map $\nu_\Gamma: \F^k \to \K^{km \times m}$ is evaluated as:
    $$\nu_\Gamma(1, \alpha) = \begin{pmatrix} \mu_\Gamma(m_1)^\top \\ \mu_\Gamma(m_\alpha)^\top \end{pmatrix}.$$
    Substituting the transposed matrices, we get 
    $$\mu_\Gamma(m_1)^\top = \begin{pmatrix} 1 & 0 \\ 0 & 1 \end{pmatrix}, \quad \mu_\Gamma(m_\alpha)^\top = \begin{pmatrix} 0 & 1 \\ 1 & 1 \end{pmatrix}.$$
    This results in the matrix
    $$\nu_\Gamma(1, \alpha) = \begin{pmatrix} 1 & 0 \\ 0 & 1 \\ 0 & 1 \\ 1 & 1 \end{pmatrix} \in \F_2^{4 \times 2}.$$
\end{example}

\begin{proposition}\label{prop:link-qsystems}
If $\mS$ is an $[n,k,d]_{\F/\K}$ system, then the set $\nu_\Gamma(\mS):=\{\nu_\Gamma(v) \, :\, v\in \mS\}\subseteq \K^{km\times m}$ is an $[m\times n,km,d]_\K$ column-system.
\end{proposition}

\begin{proof}
The set $\nu_\Gamma(\mS)$ is a $\K$-vector space of dimension $n$.
We have to show that $\sum_{M\in \nu_\Gamma(\mS)}\colsp_\K(M)$
is not contained in a hyperplane of $\K^{km}$ and that
$$\max_{\substack{\mH'\subseteq \F^k,\\ \; \dim_\F(\mH') = k-1}} \dim_\K(\sLR\cap\mH')   =
\max_{\substack{\mH\subseteq \K^{km}, \\  \dim_\K(\mH)=km-1}}\dim_\K(\nu_\Gamma(\mS)_{\ccc}(\mH)).
$$

First of all, suppose that $\sum_{M\in \nu_\Gamma(\mS)}\colsp_\K(M)$ is contained in a hyperplane $\mH$. Let us call 
$$w:=(w_{1,1},\ldots,w_{1,m},\ldots,w_{k,1},\ldots,w_{k,m})\in \K^{km}$$
a nonzero vector such that $\mH=w^\perp$. Since
$$(w_{i,1},\ldots,w_{i,m})\mu_\Gamma(m_{v_i})=0 \ \ \Longleftrightarrow \ \ \left(\sum_{j=1}^m w_{i,j}\gamma_j\right)\cdot v_i=0,$$
the hyperplane
$$\left(\sum_{j=1}^m w_{1,j}\gamma_j,\ldots,\sum_{j=1}^m w_{k,j}\gamma_j\right)^\perp\subseteq \F^k$$
contains $\mS$, contradicting the hypothesis of $\mS$ being a $\K$-system.

Secondly, with the same notation for $\mH$ and $w$, we have
\begin{align*}
\nu_\Gamma(\mS)_{\ccc}(w^\perp) & =\{\nu_\Gamma(v) \, :\, v \in \mS, \colsp_\K(\nu_\Gamma(v))\subseteq w^\perp\}\\ &=\left\{\nu_\Gamma(v) \, :\, v\in \mS\cap\left(\sum_{j=1}^m w_{1,j}\gamma_j,\ldots,\sum_{j=1}^m w_{k,j}\gamma_j\right)^\perp\right\}\\
& =\nu_\Gamma\left(\mS\cap\left(\sum_{j=1}^m w_{1,j}\gamma_j,\ldots,\sum_{j=1}^m w_{k,j}\gamma_j\right)^\perp\right),
\end{align*}
so that 
$$\max \left\{\dim_\K(\sLR\cap(w')^\perp) \, :\,   w'\in \F^k\setminus \{0\}\right\}=
\max\left\{\dim_\K(\nu_\Gamma(\mS)_{\ccc}(w^\perp))\,:\, w\in \K^{km}\setminus\{0\}\right\}.
$$ 
\end{proof}

\begin{remark}
    If $\Gamma'$ is another basis of $\F$ over $\K$, then, for all $v\in \F^k$,
    $$\nu_{\Gamma'}(v)=\left(\begin{array}{c|c|c|c}\mu_{\Gamma,\Gamma'}({\rm id})^{-1}&\phantom{\vdots}0\phantom{\vdots}&\cdots &0\\\hline0&\mu_{\Gamma,\Gamma'}({\rm id})^{-1}&\ddots&\vdots\\\hline\vdots&\ddots&\ddots &0\\\hline\phantom{\vdots}0\phantom{\vdots}&\cdots &0&\mu_{\Gamma,\Gamma'}({\rm id})^{-1}\end{array}\right) \nu_{\Gamma}(v)\mu_{\Gamma,\Gamma'}({\rm id}),$$
    where $\mu_{\Gamma,\Gamma'}({\rm id})$ is the representative matrix of the identity map with respect to the bases $\Gamma$ and $\Gamma'$.
Hence, $\nu_\Gamma(\mS)$ and $\nu_{\Gamma'}(\mS)$ are equivalent $\K$-column-systems.
\end{remark}

Note that in the proof of Proposition \ref{prop:link-qsystems} we have shown that 
$$\nu_\Gamma(\mS)_{\ccc}(\mathcal{H})=\nu_\Gamma(\mS\cap \tilde\mH),$$
where $\mH$ is the hyperplane of $\K^{km}$  orthogonal to $w$ and $\tilde\mH$ is the hyperplane of $\F^k$ orthogonal to $\left(\sum_{j=1}^m w_{1,j}\gamma_j,\ldots,\sum_{j=1}^m w_{k,j}\gamma_j\right)$. Similarly, if $\mathcal{W}=\bigcap_{i\in I} \mH_i$ is a subspace of $\K^{km}$, then one can prove that 
\begin{equation}\label{eq:weight}
\nu_\Gamma(\mS)_{\ccc}(\mathcal{W})=\nu_\Gamma(\mS\cap \tilde{\mathcal{W}}),    
\end{equation} where  $\tilde{\mathcal{W}}=\bigcap_{i\in I} \tilde{\mH}_i$. Note that in this case ${\rm codim}_\F(\tilde{\mathcal{W}})\leq {\rm codim}_\K(\mathcal{W})$.

\medskip

In the context of linear sets \cite{lunardon1999normal,polverino2010linear} and $q$-systems, it is common to refer to $\dim_{\K}(\mS\cap \tilde{\mathcal{W}})$ as the \textbf{weight of $\tilde{\mathcal{W}}$ in $\mS$}, and to denote it by $\wt_{\mS}(\tilde{\mathcal{W}})$. Note that $\tilde{\mathcal{W}}$ is an $\F$-subspace of $\F^k$, whereas $\mS$ is a $\K$-subspace of $\F^k$. The \textbf{linear set} $\mathcal{L}_\mS$ associated with $\mS$ is the subset of the projective space $\mathbb{P}(\F^k)$ consisting of all projective points corresponding to nonzero vectors of $\mS$. If $\tilde{\mathcal{W}}$ has dimension $1$ (so that it corresponds to a projective point), the weight $\wt_{\mS}(\tilde{\mathcal{W}})$ measures how many vectors of $\mS$ lie in $\tilde{\mathcal{W}}$, capturing how ``concentrated'' the linear set is at that point. For higher-dimensional subspaces, the weight measures the interaction between $\F$-subspaces and $\mS$.

For $\K$-column systems, this interaction no longer makes sense, since we do not have an extension field. However, the connection expressed in Eq.~\eqref{eq:weight} suggests that, in this case, what makes sense is to look at how many matrices have their column span contained in a given subspace. We will define and explore this concept in more detail in the next section.


\section{Evasive systems, generalized weights, and bounds}\label{sec:bounds}

In this section, we introduce and study the concept of the column-weight of a subspace within a matrix space. This generalizes the notion of the weight of subspaces in linear sets, as discussed above. We also introduce the concept of column-evasiveness, which specializes to the notion of column-scatteredness. As in the classical case (see e.g.~\cite{marino2023evasive}), these notions are closely related to certain structural properties of the associated codes, including the minimum distance, generalized weights, and various optimality conditions.

\begin{definition}\label{def:evasive}
Let $\mX\subseteq \K^{k\times m}$ and $\mathcal{W}\subseteq \K^{k}$ be two $\K$-subspaces. The \emph{column-weight} of $\mathcal{W}$ in~$\mX$ is defined as
$$\cwt_{\mX}(\mathcal{W}):=\dim_\K \mX_{\ccc}(\mathcal{W}).$$
For $1\leq h\leq r\leq m\cdot \dim_\K(\mX)$, we define $\mX$ to be $(h,r)$-\emph{column-evasive} if 
$$\cwt_{\mX}(\mathcal{W})\leq \frac{r}{m},$$
for all $h$-dimensional subspaces $\mathcal{W}$. In particular, if $h=r$, we define $\mX$ to be $h$-\emph{column-scattered}.
\end{definition}

\begin{remark}
A completely analogous discussion holds for $\K$-row-systems. In the same way, for any two $\K$-subspaces $\mY\subseteq  \K^{n\times k}$ and $\mathcal{W}\subseteq \K^{k}$ , we define the \emph{row-weight} of $\mathcal{W}$ in $\mY$ to be
$$\rwt_{\mY}(\mathcal{W}):=\dim_\K \mY_{\rrr}(\mathcal{W}).$$
For $1\leq h\leq r\leq n\cdot\dim_\K(\mY)$, we define $\mY$ to be $(h,r)$-\emph{row-evasive} if 
$$\rwt_{\mY}(\mathcal{W})\leq \frac{r}{n},$$
for all subspaces $\mathcal{W}$ of dimension $h$. In particular, if $h=r$, we define $\mY$ to be $h$-\emph{row-scattered}.
\end{remark}

We give an overview of classical coding-theory results that can be proved using the geometric approach. The first one is an elegant bound relating the parameters of a rank-metric code.

\begin{theorem}[\emph{Singleton bound} \cite{de78}] 
Let $\mC$ be an $[m\times n,k,d]_\K$ rank-metric code. Then,
$$
k\leq \max\{m,n\}(\min\{m,n\}-d+1).$$
\end{theorem}
\begin{proof}
We first suppose $\C$ to be column- and row-nondegenerate. We want to show that $k\leq\min\{m(n-d+1), n(m-d+1)\}$.
Let $\mS$ and $\mT$ be, respectively, any $[m\times n,k,d]_{\K}$ column- and row-systems associated with $\mC$.
From Definition \ref{def:systems} it follows that
    \begin{align*}
    n-d=\max\{\cwt_{\sMR}(\mH)\,:\, \mH\subseteq \K^k, \; \dim_\K(\mH)=k-1\},\\
    m-d=\max\{\rwt_{\mT}(\mH)\,:\, \mH\subseteq \K^k,  \; \dim_\K(\mH)=k-1\}.
    \end{align*}
Since the dimension of any hyperplane in $\K^k$ is $k-1$, every $\left\lfloor\frac{k-1}{m} \right\rfloor=\left\lceil \frac{k}{m}\right\rceil-1$
matrices are contained in $\mS_{\ccc}(\mH)$, for a suitable hyperplane $\mH\subseteq \K^k$, that is, $\cwt_{\sMR}(\mH)\ge \left\lceil \frac{k}{m}\right\rceil-1$. Analogously, using the row-system $\mT$, we derive $m-d\geq \left\lceil \frac{k}{n}\right\rceil-1$. 

If $\C$ is either column- or row-degenerate, we can isometrically embed it in a space of matrices with either fewer rows or fewer columns, so that the desired inequality holds even stronger.  
\end{proof}

\begin{definition}
An $[m\times n,k,d]_\K$ rank-metric code is called \emph{Maximum Rank Distance} (MRD) if 
$k=\max\{m,n\}(\min\{m,n\}-d+1)$.
\end{definition}

The \textbf{generalized weights} form another family of parameters that we can associate with a linear code, which generalizes in a way the notion of minimum distance. More precisely, instead of considering individual codewords, one studies the minimum support of subcodes of a given dimension.
For $\F$-linear Hamming-metric (LH) and $\F$-linear rank-metric (LR) codes, substantial research has been devoted to the study of generalized weights and their structural and combinatorial properties. In the recent work \cite{d2025generalized}, generalized weights of additive Hamming-metric (AH) codes have also been investigated. In particular, several equivalent definitions have been proposed in the literature, together with geometric characterizations, monotonicity, Wei-type duality theorems, and connections with code performance over erasure and wiretap channels; see e.g. \cite{wei2002generalized, gorla2021rank}. 
We introduce the following definition of generalized rank weights for matrix rank-metric  codes  and then observe that it is equivalent to existing notions \cite{ ravagnani2016generalized, martinez2017relative, britz2020wei, ghorpade2020polymatroid, gorla2021rank}. 

\begin{definition}\label{def:gen_weights_nostra}
    Let $\mC$ be an $[m\times n, k,d]_\K$ rank-metric code. For every $t\in [k]$ we define 
    the \textbf{$t$-th column-generalized (rank) weight} of $\mC$ as
    $$d^{\ccc}_t(\mC) := \min\{\dim_\K(\csupp_\rk(\mD)) \; : \; \mD\subseteq \mC, \; \dim_\K(\mD) = t\}$$
    and the \textbf{$t$-th row-generalized (rank) weight} of $\mC$ as
    $$d^{\rrr}_t(\mC) := \min\{\dim_\K(\rsupp_\rk(\mD)) \; : \; \mD\subseteq \mC, \; \dim_\K(\mD) = t\}.$$
\end{definition}

Observe that the above definitions are analogous to those proposed for $\F$-linear Hamming-metric codes (LH) \cite{helleseth1977weight, wei2002generalized} and rank-metric (LR) codes \cite{oggier2012existence}. Moreover, with this definition, one can read nondegeneracy from the last generalized weight, as it happens for the above-mentioned cases. In particular, if $\C$ is an $[m\times n,k]_{\K}$ code, it is easy to see that:
\begin{enumerate}
    \item $\C$ is column-nondegenerate if and only if $d_k^{\rrr}(\C)=n$;
    \item $\C$ is row-nondegenerate if and only if $d_k^{\ccc}(\C)=m$.
\end{enumerate}

We point out that in \cite{martinez2017relative, britz2020wei,ghorpade2020polymatroid}, only column-generalized weights are defined and used, while in
\cite{gorla2021rank, ravagnani2016generalized} for matrices in $\K^{m\times n}$ row-generalized weights are considered when $m<n$ and column-generalized weights when $m\geq n$. In our framework, it is extremely natural to consider both, since they provide different information on the code; see Remark \ref{rem:gen_weights}.

\begin{example}\label{exa:genweights}
Let $\mC$ be the $[3\times 3,4,2]_2$ code defined in Example \ref{ex:toy}. In this case, one can verify that
$$d_1^{\ccc}(\mC)=2, \ d_2^{\ccc}(\C)=2, \ d_3^{\ccc}(\C)=3 , \ d_4^{\ccc}(\C)=3$$
and
$$d_1^{\rrr}(\mC)=2, \ d_2^{\rrr}(\C)=2, \ d_3^{\rrr}(\C)=3 , \ d_4^{\rrr}(\C)=3.$$
\end{example}

\begin{remark}\label{rem:gen_weights}
In Example \ref{exa:genweights}, it turns out that the sequences of column- and row-generalized weights  coincide. However, we remark that this is not always the case. In particular, if $n\neq m$ and the code $\C$ is column- and row-nondegenerate, it is clear that the two sequences cannot coincide, since $m=d_k^{\ccc}(\C)\neq d_k^{\rrr}(\C)=n$.
\end{remark}

In \cite{britz2020wei}, the generalized weights of an $[m\times n, k,d]_\K$ rank-metric code are defined as follows. 

\begin{definition}[\cite{britz2020wei}]\label{def:gen_weights_britz}
     Let $\mC$ be an $[m\times n, k,d]_\K$ code. For every $t\in [k]$ we define the \textbf{$t$-th generalized weight} of $\mC$ as
     $$D_t(\mC) := \min\{\dim_\K(\mU) \; : \; \mU\subseteq \K^m, \;  \dim_\K(\mC_{\ccc}(\mU))\geq t\}.$$
\end{definition}

In the following result, we show that the column-generalized weights in
Definition~\ref{def:gen_weights_nostra} coincide with the generalized weights of Definition~\ref{def:gen_weights_britz}.
The same argument gives an analogous characterization of the row-generalized
weights. Moreover, the geometric viewpoint clarifies that both notions of row- and column-generalized weights should be considered. 

\begin{theorem}\label{thm:same_gen_weights}
 Let $\mC$ be an $[m\times n, k,d]_\K$ rank-metric code. For every $t\in [k]$ we have
 $$d_t^{\ccc}(\C)=\min\{\dim_\K(\mU) \; : \; \dim_\K(\mC_{\ccc}(\mU))\geq t\}$$
 and 
  $$d_t^{\rrr}(\C)=\min\{\dim_\K(\mU) \; : \; \dim_\K(\mC_{\rrr}(\mU))\geq t\}.$$
  In particular, $d_t^{\ccc}(\C) = D_t(\mC)$.
\end{theorem}
\begin{proof}
First, let $\mD\subseteq\mC$ be a $t$-dimensional subcode. Let $$\mU:=\csupp_{\rm rk}(\mD)=\sum_{M\in \mD}\colsp_\K(M)\subseteq \K^m.$$
Then every codeword of $\mD$ has column space contained in $\mU$, hence $\mD\subseteq \mC_{\ccc}(\mU)$.
Therefore, $\dim_{\K}(\mC_{\ccc}(\mU))\geq t$, and so $D_t(\mC)\leq \dim_\K(\mU)=\dim_\K(\csupp_{\rm rk}(\mD))$.
Taking the minimum over all $t$-dimensional subcodes $\mD\subseteq\mC$, we get $D_t(\mC)\leq d_t^{\ccc}(\mC)$.

Conversely, let $\mU\subseteq \K^m$ be such that $\dim_\K(\C_{\ccc}(\mU))\geq t$. Choose a
$t$-dimensional subspace $\mD\subseteq\mC_{\ccc}(\mU)$. Since every $M\in\mD$ satisfies
$\colsp_\K(M)\subseteq \mU$, we have $\csupp_{\rm rk}(\mD)\subseteq \mU$.
Thus,
$d_t^{\ccc}(\mC)\leq \dim_\K(\csupp_{\rm rk}(\mD))\leq \dim_\K(\mU)$. Taking the minimum over all such $\mU$, we obtain $d_t^{\ccc}(\mC)\leq D_t(\mC)$.
Hence, $d_t^{\ccc}(\C) = D_t(\mC)$.

The proof of the row version is identical, replacing column spaces by row spaces.
\end{proof}

\begin{remark}
If $\C$ is an $[n,k,d]_{\F/\K}$ code, then the generalized weights are defined \cite{oggier2012existence} as
$$d_t(\C)=\min \{\dim_\K(\supp_\rk(\mathcal{D})) \ : \ \mathcal{D}\subseteq \mC, \ \dim_\F(\mathcal{D})=t\}$$
for $t\in [k]$.  

Note that here $\mC$ is an $\F$-subspace of $\F^n$ and the spaces considered for computing $d_t(\mC)$ are $\F$-subspaces of $\C$. Let $\Gamma=\{\gamma_1,\ldots,\gamma_m\}$ be a basis of $\F$ over $\K$ and consider $\Gamma(\C)\subseteq \K^{m\times n}$. Then, $\Gamma(\mC)$ is an $[m\times n,km,d]_\K$ code. Hence,
$$d_{mt-i+1}^{\rrr}(\Gamma(\C))\leq d_t(\mC),$$
for all $i\in [m]$, $t\in [k]$.

Note that also the column-generalized weights of $\Gamma(\C)$ can be related to some generalized weights of $\mC$, defined by taking as support of $(v_1,\ldots,v_n)\in \F^n$ the $\K$-subspace $\langle v_1,\ldots,v_n\rangle_\K$ of $\F$. To the best of our knowledge, it seems that no one has used them so far in the study of $\F$-linear rank-metric codes.
\end{remark}

\begin{remark}\label{rmk:corr-supp}
Let $\C$ be a column-nondegenerate $[m\times n,k,d]_K$ code, and let
$\mS$ be any column-system associated with $\C$. Thus, up to
replacing $\mS$ by an equivalent representative, let $\mS=\sss_3(T)$ for a suitable generator tensor $T$ for $\mC$.

For every subspace $\mU\subseteq \K^k$, set
\[
\mD_\mU:=\{m_1(T,u) \, :\, u\in \mU\}\subseteq \C .
\]
By Lemma~\ref{lem:higher-support-column-system}, we have
\[
\rsupp_{\rk}(\mD_\mU)^\perp
=
\psi_T^{-1}\bigl(\mS_{\ccc}(\mU^\perp)\bigr),
\]
where
\[
\psi_T:\K^n\longrightarrow \mS,\ \ \ v\longmapsto m_3(T,v).
\]
In particular,
\[
\dim_{\K}(\mS_{\ccc}(\mU^\perp))
=
n-\dim_\K(\rsupp_{\rk}(\mD_\mU)).
\]

Since the map $u\mapsto m_1(T,u)$ is an isomorphism from $\K^k$ onto $\C$, the $t$-dimensional subcodes of $\C$ are precisely the spaces
$\mD_\mU$, with $\mU\subseteq \K^k$ and $\dim_\K(\mU)=t$. Therefore
\[
d_t^{\rrr}(\C)=\min_{\substack{\mU\subseteq \K^k\\ \dim_\K(\mU)=t}}
\dim_\K(\rsupp_{\rk}(\mD_\mU)).
\]
Using the dimension identity above, we obtain
\[
n-d_t^{\rrr}(\C)
=
\max_{\substack{\mU\subseteq \K^k\\ \dim_\K(\mU)=t}}
\dim_\K(\mS_{\ccc}(\mU^\perp)).
\]
Equivalently, after setting $\mathcal W=\mU^\perp$, this becomes
\begin{equation}\label{eq:gen-row-cwt}
n-d_t^{\rrr}(\C)
=
\max\{ \cwt_S(\mathcal W) \, :\, \mathcal W\subseteq \K^k,\ \dim_\K (\mathcal W)=k-t\}.
\end{equation}

Similarly, if $\C$ is row-nondegenerate and $\mT$ is a
row-system associated with $\C$, then, for a suitable generator tensor $T$, we have $\mT=\sss_2(T)^\top$. The row-system
analogue of Lemma~\ref{lem:higher-support-column-system} gives
\begin{equation}\label{eq:gen-col-rwt}
m-d_t^{\ccc}(\C)
=
\max\{ \rwt_{\mathcal T}(\mathcal W) \, :\, \mathcal W\subseteq \K^k,\ \dim_\K(\mathcal W)=k-t\}.
\end{equation}
\end{remark}

We now use the characterization of row- and column-generalized weights of a nondegenerate rank-metric code given in Eq. \eqref{eq:gen-row-cwt} and Eq. \eqref{eq:gen-col-rwt}, and relate them to the evasiveness properties of the associated column- and row-systems.
\begin{theorem}\label{thm:genweights-evasive}
    Let $\C$ be an $[m\times n,k,d]_\K$  code and let $t\in [k]$.
    \begin{enumerate}
        \item Assume that $\C$ is column-nondegenerate and let $\mS$ be any $\K$-column-system associated with~$\C$. Then, for every \(s\in[n]\), the following are equivalent.
        \begin{enumerate}
            \item $d_t^{\rrr}(\C)\geq s$.
            \item $\mS$ is $(k-t,m(n-s))$-column-evasive.
        \end{enumerate}
        \item Assume that $\C$ is row-nondegenerate and let $\mT$ be any $\K$-row-system associated with $\C$. Then, for every \(s\in[m]\), the following are equivalent.
        \begin{enumerate}
            \item $d_t^{\ccc}(\C)\geq s$.
            \item $\mT$ is $(k-t,n(m-s))$-row-evasive.
        \end{enumerate}
    \end{enumerate}
\end{theorem}
\begin{proof}
    Let us prove part (1), since part (2) is completely analogous.  By Eq. \eqref{eq:gen-row-cwt}, $d_t^{\rrr}(\C)\geq s$ if and only if $n-s\ge \cwt_{\sMR}(\mathcal{W})$ for every $\mathcal{W}\subseteq \K^k$ with  $\dim_\K(\mathcal{W})=k-t$. Hence, $\cwt_{\sMR}(\mathcal{W})\le \frac{m(n-s)}{m}$ for every $\mathcal{W}\subseteq \K^k$ with  $\dim_\K(\mathcal{W})=k-t$, which means that $\mS$ is $(k-t,m(n-s))$-column-evasive.
\end{proof}

Using the geometric characterization of column- and row-generalized weights observed in Remark~\ref{rmk:corr-supp} and formalized in Theorem \ref{thm:genweights-evasive}, we can deduce the following upper bound, originally shown in \cite{britz2020wei} (only for column-generalized weights).

\begin{theorem}[{\cite[Theorem 10]{britz2020wei}}]\label{thm:bounds_genweights}
     Let $\mC$ be an $[m\times n, k,d]_\K$ code. For every $t\in[k]$, we have 
     $$d^{\rrr}_t(\mC) \leq n-\left\lfloor  \frac{k-t}{m}\right\rfloor  \ \ \ \text{and} \ \ \
     d^{\ccc}_t(\mC) \leq m-\left\lfloor  \frac{k-t}{n}\right\rfloor.$$
\end{theorem}
\begin{proof}

We first suppose $\C$ to be column- and row-nondegenerate.\\
Let $\mS$ and $\mT$ be, respectively, any $[m\times n,k,d]_{\K}$ column- and row-systems associated with $\mC$.
By Remark \ref{rmk:corr-supp}, 
    \begin{align*}
   n-d^{\rrr}_t(\C)=\max\{\cwt_{\sMR}(\mathcal{W})\,:\, \mathcal{W}\subseteq \K^k,  \dim_\K(\mathcal{W})=k-t\},\\
m-d^{\ccc}_t(\C)=\max\{\rwt_{\mathcal{T}}(\mathcal{W})\,:\, \mathcal{W}\subseteq \K^k,  \dim_\K(\mathcal{W})=k-t\}.
    \end{align*}
Every $\left\lfloor\frac{k-t}{m} \right\rfloor$
matrices are contained in $\mS_{\ccc}(\mathcal{W})$ for a suitable subspace $\mathcal{W}\subseteq \K^k$ of codimension $t$. This means that $\cwt_{\sMR}(\mathcal{W})\ge \left\lfloor \frac{k-t}{m}\right\rfloor$. Analogously, using the row-system $\mT$, we derive the other inequality.

Again, if $\C$ is either column- or row-degenerate, we can isometrically embed it in a space of matrices with either fewer rows or fewer columns, so that the desired inequalities hold even stronger.
\end{proof}

We now conclude this section characterizing rank-metric codes meeting the bounds in Theorem \ref{thm:bounds_genweights} with equality, in terms of the evasiveness properties of the associated $\K$-column- and $\K$-row-systems.

\begin{theorem}
Let $\C$ be an $[m\times n,k,d]_\K$ code with generator tensor $T$. The following holds for every $t\in [k]$:
\begin{enumerate}
    \item If $\mC$ is column-nondegenerate and $n\leq m$, then $d_t^{\rrr}(\mC)=n-\left\lfloor  \frac{k-t}{m}\right\rfloor$ if and only if $\sss_3(T)$ is $(k-t)$-column-scattered.
    \item If $\C$ is row-nondegenerate and $m\leq n$, then $d_t^{\ccc}(\mC)=m-\left\lfloor  \frac{k-t}{n}\right\rfloor$ if and only if $\sss_2(T)^\top$ is $(k-t)$-row-scattered.
\end{enumerate}
In particular, MRD codes correspond to $(k-1)$-column-scattered $\K$-column-systems, if $n\leq m$, and to $(k-1)$-row-scattered $\K$-row-systems, if $m\leq n$. 
\end{theorem}

\begin{proof}
It is a straightforward consequence of Theorems \ref{thm:genweights-evasive} and \ref{thm:bounds_genweights}, and Definition \ref{def:evasive}.
\end{proof}

\begin{remark}
 Examples of $(k-1)$-column-scattered $\K$-column-systems can be obtained from known constructions of MRD codes, such as Gabidulin codes or their generalizations (see, for instance, \cite{lunardon2018generalized}). When these codes are not linear over the extension field, the resulting $\K$-column-systems are genuinely $(k-1)$-column-scattered: they are not obtained in a straightforward way from $q$-systems with similar properties, but instead exhibit the scatteredness intrinsically within the $\K$-linear framework.
   
\end{remark}

\section{Faithful and one-weight codes}\label{sec:conseq}

In this section, we give an overview on some interesting consequences of Theorem \ref{thm:StandardEq}, regarding \textit{faithful} and \textit{one-weight codes}. Throughout this section, we only consider finite fields, hence, we fix  $$\boxed{\K=\Fq,} \quad \textnormal{ and } \quad \boxed{\F=\F_{q^m}.}$$

\subsection{Faithful codes}

We introduce and study the notion of faithful codes. This notion is inspired by the notion of faithful additive codes in the Hamming metric, which has several interesting implications; see e.g. \cite{ball2025griesmer, kurz2024additive}.

We start with a preliminary result, which is a direct consequence of Theorem \ref{thm:StandardEq}.
 \begin{corollary}\label{cor:bound_sum}
     Let $\mC$ be an $[m\times n,k,d]_q$ code. Then 
     $$\sum_{M \in \mC^*}(q^{n-\rk(M)}-1)\ge (q^{k-m}-1)(q^n-1),$$ 
          $$\sum_{M \in \mC^*}(q^{m-\rk(M)}-1)\ge (q^{k-n}-1)(q^m-1).$$
 \end{corollary}

 \begin{proof} 
    By symmetry, we just prove the first inequality. 
     First, assume that $\C$ is  column-nondegenerate, and let $\mS$ be any of its associated column-systems. First of all we have that $\dim_{\Fq}(\mS)=n$. Then, the statement directly follows from Theorem~\ref{thm:StandardEq}, since $\rk(X)\le m$ for every $X\in \mS^*$. 
     
     If $\C$ is column-degenerate, then $\C$ can be isometrically embedded in a smaller matrix space $\Fq^{m\times n'}$ with $n=n'+b$, $b>0$. Thus, we can use the previous argument to show the statement for $m$ and $n'$, that is,

     $$\sum_{M \in \mC^*}(q^{n-b-\rk(M)}-1)\ge (q^{k-m}-1)(q^{n-b}-1),$$ 
     Thus,
       $$\sum_{M \in \mC^*}q^{n-b-\rk(M)}\ge (q^{k-m}-1)(q^{n-b}-1)+q^k-1=q^{k-m+n-b}-q^{k-m}-q^{n-b}+q^k,$$ 
and, multiplying both sides by $q^b$, we get
       \begin{align*}
           \sum_{M \in \mC^*}q^{n-\rk(M)}&\ge q^{k-m+n}-q^{k-m+b}-q^{n}+q^{k+b}=q^{k-m+n}-q^{n}+q^b(q^{k}-q^{k-m})\\
           &\geq q^{k-m+n}-q^{n}+q^{k}-q^{k-m}=(q^{k-m}-1)(q^{n}-1)+q^k-1,
        \end{align*}
        proving the first desired inequality.\\The same  argument using the associated row-system shows the second inequality.
 \end{proof}

We now give the notion of faithfulness, by distinguishing a priori the column- and the row-faithfulness of a code.

\begin{definition}\label{def:c-r-faithful}
    Let $\mC$ be an $[m\times n,k,d]_q$  code and let $T\in \Fq^{k\times m \times n}$ be a generator tensor of $\C$. 
    We say that $\mC$ is \textbf{column-faithful} if $\rk(X)=m$ for every  $X \in \sss_3(T)\setminus\{0\}$, and that it is \textbf{row-faithful} if $\rk(Y)=n$ for every $Y \in \sss_2(T)\setminus\{0\}$.\footnote{Note that this definition can be given over any field $\K$.}
\end{definition}

First of all, observe that for a code to be column-faithful, one automatically needs $k\ge m$. Similarly, for a code to be row-faithful, one needs $k\ge n$. Moreover, from Definition \ref{def:c-r-faithful}, a column-faithful code is automatically row-nondegenerate, while a row-faithful code is automatically column-nondegenerate.

However, another consequence of Theorem \ref{thm:StandardEq} is that the two notions of column and row faithfulness coincide, implying also that faithful codes are both column- and row-nondegenerate. This is shown in the following result.

 \begin{theorem}\label{thm:faithfulness}
     Let $\mC$ be an $[m\times n,k,d]_q$ code. Then, the following are equivalent
     \begin{enumerate}
         \item $\mC$ is column-faithful.
         \item   $$\sum_{M \in \mC^*}q^{n-\rk(M)}= (q^{k-m}-1)(q^n-1)+q^k-1.$$ 
         \item $\mC$ is row-faithful.
         \item  $$\sum_{M \in \mC^*}q^{m-\rk(M)}= (q^{k-n}-1)(q^m-1)+q^k-1.$$
     \end{enumerate}
 \end{theorem}

 \begin{proof}
 The equivalence between (1) and (2) follows immediately from Corollary \ref{cor:bound_sum}, where equality holds if and only if $\rk(X)=m$ for every $X \in \mS^*$, that is, if and only if $\mC$ is column-faithful. The same holds for the equivalence between (3) and (4).

 We now show that (2) implies (4). Assume that 
    $$\sum_{M \in \mC^*}q^{n-\rk(M)}= (q^{k-m}-1)(q^n-1)+q^k-1.$$ 
   Thus, we can write
     \begin{align*}\sum_{M \in \mC^*}q^{m-\rk(M)}&=q^{m-n}\sum_{M \in \mC^*}q^{n-\rk(M)}\\
     &=q^{m-n}((q^{k-m}-1)(q^n-1)+q^k-1) \\
     &= q^{m-n}(q^{k-m+n}-q^n-q^{k-m}+q^k)\\
     &=q^k-q^m-q^{k-n}+q^{k+m-n}\\&=
     (q^{k-n}-1)(q^m-1)+q^k-1.
     \end{align*}
     Clearly, also the implication (4) $\Longrightarrow$ (2) follows from the discussion above, since only equalities are involved.
 \end{proof}

    In light of Theorem \ref{thm:faithfulness}, the notions of column-faithful code and row-faithful code coincide, and we can therefore define what a faithful code is, without distinguishing between column- and row-faithfulness.

\begin{definition}
An $[m\times n,k,d]_q$ code is said to be \textbf{faithful} if it is column-faithful (or, equivalently, row-faithful).
\end{definition}

\begin{remark}
    Faithful codes represent the \textit{dual} counterpart of one-weight rank-metric codes whose only nonzero weight is the maximum achievable one. Here, by dual counterpart we mean that the {\it duality} is given by the geometric point of view of passing from a rank-metric code to a column-system via $\phi^{\mathrm{MR_{\ccc}}}$ and vice-versa. 
\end{remark}

As a consequence of Theorem~\ref{thm:faithfulness}, we can give a lower bound on the number of codewords of maximum rank in a faithful code.

\begin{theorem}\label{thm:Amin2}
Let $\mC$ be a faithful $[m\times n,k,d]_q$ code, and let
$(A_i(\mC))_{0\leq i \leq \min\{m,n\}}$ be its rank distribution. Then
$$
A_{\min\{m,n\}}(\mC)
\geq
\frac{
q(q^k-1)-\bigl((q^{k-\max\{m,n\}}-1)(q^{\min\{m,n\}}-1)+q^k-1\bigr)
}{q-1}.$$
\end{theorem}

\begin{proof}
Since $\mC$ is faithful, Theorem~\ref{thm:faithfulness} gives
$$
\sum_{M\in \mC^*} q^{\min\{m,n\}-\rk(M)}
=
(q^{k-\max\{m,n\}}-1)(q^{\min\{m,n\}}-1)+q^k-1.
$$
Indeed, if $m\leq n$, this is the row-faithful identity from Theorem~\ref{thm:faithfulness}(4), while if
$n\le m$, it is the column-faithful identity from Theorem~\ref{thm:faithfulness}(2).

For simplicity, write $A_i=A_i(\mC)$. Then
$$
\sum_{i=1}^{\min\{m,n\}} A_i q^{\min\{m,n\}-i}
=
(q^{k-\max\{m,n\}}-1)(q^{\min\{m,n\}}-1)+q^k-1.
$$
Moreover, we have that 
$$
\sum_{i=1}^{\min\{m,n\}} A_i=q^k-1.
$$
Since for every $i<\min\{m,n\}$ it holds
$$
q^{\min\{m,n\}-i}\ge q,
$$
we have
$$
\sum_{i=1}^{\min\{m,n\}} A_i q^{\min\{m,n\}-i}
\geq
A_{\min\{m,n\}}+
q\sum_{i=1}^{\min\{m,n\}-1}A_i,
$$
and hence
$$
\sum_{i=1}^{\min\{m,n\}} A_i q^{\min\{m,n\}-i}
\ge
A_{\min\{m,n\}}+
q\bigl(q^k-1-A_{\min\{m,n\}}\bigr).
$$
Therefore,
$$
(q^{k-\max\{m,n\}}-1)(q^{\min\{m,n\}}-1)+q^k-1
\geq
q(q^k-1)-(q-1)A_{\min\{m,n\}}.
$$
Rearranging gives
$$
A_{\min\{m,n\}}(\C)
\geq
\frac{
q(q^k-1)-\bigl((q^{k-\max\{m,n\}}-1)(q^{\min\{m,n\}}-1)+q^k-1\bigr)
}{q-1}.
$$
\end{proof}

In particular, the bound in Theorem~\ref{thm:Amin2} is strong enough to force the existence of a
maximum-rank codeword in every faithful code.

\begin{corollary}\label{cor:min_codewords}
Let $\mC$ be a faithful $[m\times n,k,d]_q$ code. Then $\mC$ contains
a codeword of rank $\min\{m,n\}$, that is, $A_{\min\{m,n\}}(\C)>0$.
\end{corollary}

\begin{proof}
By Theorem~\ref{thm:Amin2},
$$
A_{\min\{m,n\}}(\mC)
\ge
\frac{
(q-1)q^k-q-q^{k-\max\{m,n\}+\min\{m,n\}}
+q^{k-\max\{m,n\}}+q^{\min\{m,n\}}
}{q-1}.
$$
Since $\mC$ is faithful, necessarily $k\geq \max\{m,n\}$. Hence the
numerator is positive for every $q\geq 2$. Therefore $A_{\min\{m,n\}}(\C)>0$.
Finally, notice that the numerator is divisible by $q-1$, since modulo $q-1$ every
power of $q$ is congruent to $1$.
Hence the right-hand side is an integer.
\end{proof}

\begin{remark}
Corollary~\ref{cor:min_codewords} shows that every faithful matrix rank-metric code contains a codeword of maximum possible rank, for every finite field. This extends, in the present matrix-code setting, the analogous existence result known
for nondegenerate $\F_{q^m}$-linear rank-metric codes from \cite[Proposition 3.11]{alfarano2021linear}; see \cite{alfarano2021linear, alfarano2026recursive} for further implications of this result. 
\end{remark}

\subsection{One-weight codes}

In this section, we use  Theorem \ref{thm:StandardEq} to derive results on one-weight $[m\times n,k]_q$ codes. One-weight codes in the rank metric have been intensively studied in different settings, with arguments arising from representation theory, algebraic geometry, character theory and computer search; see e.g. \cite{dumas2010subspaces, boston2010spaces, boralevi2013linear, landsberg2024equivariant}.

The following bound on the dimension of a one-weight code is not new, and was originally discovered in \cite{dumas2010subspaces}. We provide an alternative proof for it based on our setting, in order to illustrate the multiple implications of Theorem \ref{thm:StandardEq}.

\begin{theorem}[{\cite[Theorem 6]{dumas2010subspaces}}]
\label{thm:bound_one_weight}
    Let $\mC$ be an $[m\times n,k,d]_q$ one-weight code. Then
    $k\le m+n-d$.
\end{theorem}

\begin{proof}
Without loss of generality, we can assume $\mC$ to be nondegenerate. Let $\mS$ be any of its associated  $[m\times n,k,d]_q$ column-system. Then, by Theorem \ref{thm:StandardEq}, we have
$$\sum_{X\in \mS^*}q^{k-\rk(X)}-(q^n-1)=\sum_{X\in\mS^*}(q^{k-\rk(X)}-1)=(q^{n-d}-1)(q^k-1)=q^{n-d+k}-q^k-q^{n-d}+1,$$
    from which we derive 
    $$\sum_{X\in \mS^*}q^{k-\rk(X)}=q^{n-d+k}+q^n-q^k-q^{n-d}.$$
If $k\le m$, then there is nothing to prove. Thus, assume $k>m$ and consider the above equation modulo $q^{k-m}$. Since $k-\rk(X)\ge k-m$ for every nonzero $X \in \mS$, also the right-hand side of the above identity must be $0$, which means
$$q^{n-d+k}+q^n-q^k-q^{n-d} \equiv 0 \pmod{q^{k-m}}.$$
Thus,  $q^{n-d}(q^d-1)\equiv 0 \pmod{q^{k-m}}$ and, since $k-m>0$, 
$q^{n-d}\equiv 0 \pmod{q^{k-m}}$,
from which we deduce $n-d\geq k-m$.
\end{proof}

\begin{proposition}
    Let $\C$ be an $[m\times n,k,d]_q$ faithful one-weight code. Then,
    $k= \max\{m,n\}$ and $d=\min\{m,n\}$.
\end{proposition}

\begin{proof}
Since $\C$ is a $[m\times n,k,d]_q$ one-weight code, then 
$$\sum_{M\in\C^*}(q^{n-\rk(M)}-1)=(q^{n-d}-1)(q^k-1).$$
On the other hand, since $\C$ is 
faithful, we must have
$$\sum_{M \in \mC^*}(q^{n-\rk(M)}-1)= (q^{k-m}-1)(q^n-1).$$ 
Putting these two equations together, we obtain
\begin{equation}\label{eq:oneweight-faithful}
    (q^{n-d}-1)(q^k-1)=(q^{k-m}-1)(q^n-1).
\end{equation}
Now, recall that, since $\mC$ is faithful, then $k\geq\max\{m,n\}$, and, in particular, $k\geq m$.
If $k=m$, then Eq. \eqref{eq:oneweight-faithful} automatically implies $d=n$, and clearly this implies $n=\min\{m,n\}$ and $m=\max\{m,n\}$.
Assume now that $k>m$. By \cite[Lemma 3.15]{alfarano2021linear}, Eq. \eqref{eq:oneweight-faithful} implies $\{n-d,k\}=\{k-m,n\}$, that is, $n-d=k-m$ and $k=n$. Thus, $k=n=\max\{m,n\}$ and $d=m=\min\{m,n\}$. 
\end{proof}

 \begin{corollary}\label{cor:rank-metric-maxrank}
 Any $[m\times n,k,n]_q$ code $\mC$ is column-nondegenerate, and, if $\mS$ is any of its associated $[m\times n,k,n]_q$ column-systems, then $\mS$ is also a $[k\times m,n,k]_q$ code. Analogously, any $[m\times n,k,m]_q$ code $\mC$ is row-nondegenerate, and, if $\mT$ is any of its associated $[m\times n,k,m]_q$ row-systems, then $\mT$ is also a $[n\times k,m,k]_q$ code.
 \end{corollary}
 
 \begin{proof}
    We prove only the first statement, since the second one can be proved analogously. 
    
     Since $d=n$, then $\C$ is a one-weight column-nondegenerate code. Let $\mS$ be any of its associated column-systems. We can look at  $\mS$ as a $[k\times m,n]_q$ code, and use Theorem \ref{thm:StandardEq}. We deduce that
     $$0=\sum_{M\in \C^*}(q^{n-\rk(M)}-1)=\sum_{X\in \mS^*}(q^{k-\rk(X)}-1),$$
     which implies that $\rk(X)=k$ for every $X\in \mS^*$. Thus, the minimum rank of $\mS$ is $k$.
 \end{proof}

 In particular, using  Corollary \ref{cor:rank-metric-maxrank} twice in the square case, we deduce a result of Knuth \cite{knuth1965finite}.
 
 \begin{corollary}[{\cite{knuth1965finite}}]\label{cor:knuth}
     Any $[n\times n,n,n]_q$ code $\C$  is both column- and row-nondegenerate, and, if $\mS$ and $\mT$ are, respectively, any of the associated column-systems and row-systems, then they are both $[n\times n,n,n]_q$ codes.
 \end{corollary}

The result of Knuth \cite{knuth1965finite} is stated in the context of  finite semifields. It is known that finite semifields that are $n$-dimensional $\Fq$-algebras correspond to $[n\times n,n,n]_q$ codes. The result of Knuth states that from a semifield, one has a full orbit of (up to) six semifields -- known nowadays as the  \textbf{Knuth orbit}. These six elements are obtained, in the language of $[n\times n,n,n]_q$ codes by combining the three operations of transposition ($\cdot^\top$), passing to (one of) the associated column-system ($\phi^{\rm MR_{\ccc}}(\cdot)$) and to (one of) the associated row-system ($\phi^{\rm MR_{\rrr}}(\cdot)$).

\medskip

It would be interesting to understand which of the results of this section remain true when $\K$ is infinite. In particular, one may ask the following questions.
\begin{enumerate}
    \item Let $\mC$ be an $[m\times n,k,d]_\K$ code. Is row-faithfulness equivalent to column-faithfulness? 
    \item Let $\mC$ be an $[m\times n,k,d]_\K$ one-weight code. Does the inequality
    $k\le m+n-d$ from Theorem~\ref{thm:bound_one_weight} hold?   
    We point out that in 1990, Beasley and Laffey claimed that a one-weight $[m\times n,k,d]_\K$ code satisfies $k\leq n$, assuming $m\leq n$ and $|\K|\geq d+1$; see \cite{beasley1990linear}. In a subsequent erratum, published in 1993, the authors noted that their theorem had only been proved under the additional assumption $n\geq 2d-1$.
\end{enumerate}

\section{Extended codes and their relations}\label{sec:extended-Hamming}

In analogy to \cite[\S4.2]{alfarano2021linear}, we aim to associate a Hamming-metric additive code to a given rank-metric code. The idea is that this (exponentially long) Hamming-metric code retains precise information
about several metric and geometric properties of the original one. Since such a result only holds over finite fields, we will only work in the case $$\boxed{\K=\Fq.}$$

\medskip

We define a map between
$\SMRc$ and $\mathscr{S}_{\rm AH}(\frac{q^n-1}{q-1},k)$ as follows. For a column-system $\sMR\in \SMRc$,  let $\mathcal{R}(\sMR)\subseteq \mS^\ast$
be a set of representatives of the one-dimensional subspaces of~$\mS$, that is, a set of representatives of $\mS^\ast/\mathbb \F_q^\ast$. Then
$|\mathcal R(\mS)|=(q^n-1)/(q-1).$
We define the multiset of projective subspaces of $\mathbb{P}(\F_q^k)$
\[\mS_{\ccc}^{\rm AH}:=\{\mathbb{P}(\colsp_{\Fq}(X)) \; : \; X\in \mR(\mS)\}.\]
Note that, if we consider two equivalent column-systems $\mX$ and $\mathcal{Y}$, then the $m$-projective systems $\mX_{\ccc}^{\rm AH}$ and $\mY_{\ccc}^{\rm AH}$ are equivalent, according to the notions of equivalence presented in Subsection~\ref{subsec:geo_tutto}. Therefore, we have a well-defined map
\[{\rm Ext_{\ccc}^{AH}}:\SMRc/_\sim\to \mathscr{S}_{\rm AH}\left(\frac{q^n-1}{q-1},k\right)/_\sim, \ \ \ [\mS]\mapsto [\mS_{\ccc}^{\rm AH}].\]
Furthermore, recall that the maps
\[\phi^{\rm MR_{\ccc}}:\FMRc/_\sim \to \SMRc/_\sim\] 
and  
\[\phi^{\rm AH}:\mathscr{F}_{\rm AH}\left(\frac{q^n-1}{q-1},k\right)/_\sim\to \mathscr{S}_{\rm AH}\left(\frac{q^n-1}{q-1},k\right)/_\sim\]
are one-to-one correspondences (see Theorems~\ref{thm:correspondence} and~\ref{thm:correspondences}). 

We have now all the tools to define an associated Hamming-metric code.

\begin{definition}\label{def:column_Hamming}
For a given code $\mC\in \FMRc$, a \emph{column-Hamming-metric code} $\mC^{\rm AH}\in \mathscr{F}_{\rm AH}(\frac{q^n-1}{q-1},k)$ \emph{associated with} $\mC$ is any code in the equivalence class $((\phi^{\rm AH})^{-1}\circ{\rm Ext_{\ccc}^{AH}}\circ\phi^{\rm MR_{\ccc}})([\mC])$.   
\end{definition}

There is a relation between the weight distributions of an $[m\times n,k]_{q}$ code and any of its associated column-Hamming-metric codes, as the following theorem shows.

\begin{theorem}\label{thm:extended}
Let $\mC\in \FMRc$ and $\mC^{\rm AH}\in \mathscr{F}_{\rm AH}(\frac{q^n-1}{q-1},k)$ be a column-Hamming-metric code  associated with $\mC$. For every generator tensor $T\in \F_q^{k\times m\times n}$ of $\mC$ one can choose a representative
of the equivalence class of $\mC^{\rm AH}$ with generator matrix
$G^{\rm AH}\in \F_{q^m}^{k\times \frac{q^n-1}{q-1}}$, such that for every $u\in \F_q^k$ we have
$$\boxed{\wt_{\HH}(uG^{\rm AH})=\frac{q^n-q^{n-\rk(m_1(T,u))}}{q-1}.}$$
In particular, if $\mC$ is a column-nondegenerate $[m\times n,k,d]_q$ code, $\mC^{\rm AH}$ is a $[\frac{q^n-1}{q-1},\frac{k}{m},\frac{q^n-q^{n-d}}{q-1}]_q^m$ additive code over $\F_{q^m}$.
\end{theorem}

\begin{proof}
Let $u\in\F_q^k$. Recall that
$$\rk(m_1(T,u))=n-\dim_\K\sMR_{\ccc}( u^\perp),$$
where $\mS=\sss_3(T)$.
Now
$$\sMR_{\ccc}( u^\perp)=\{X\in \mS   \; : \; \colsp_{\F_q}(X)\subseteq  u^\perp\},$$
so that
$$\rk(m_1(T,u))=n-\dim_{\F_q}(\sMR_{\ccc}( u^\perp))=n-\log_q(|\{X\in \mS  \; : \; \colsp_{\F_q}(X)\subseteq u^\perp\}|).$$
Hence 
$$|\{X\in \mS:\colsp_{\F_q}(X)\subseteq  u^\perp\}|=q^{n-\rk(m_1(T,u))}.$$
Let $\mathcal R(\mS)=\left\{G_1,\ldots,G_{\frac{q^n-1}{q-1}}\right\}$
be a set of representatives of \(\mS^*/\mathbb \F_q^*\). Then, as a multiset, 
$$\mS_{\ccc}^{\rm AH} = \left\{\mathbb{P}(\colsp_{\F_q}(G_i))  \; : \; i\in\left [\frac{q^n-1}{q-1} \right]\right\}.$$
Fix an $\Fq$-basis $\Gamma$ of $\mathbb F_{q^m}$. For each $i\in\left[\frac{q^n-1}{q-1}\right]$, let $g_i\in \mathbb F_{q^m}^k$ be the vector whose expansion
with respect to $\Gamma$ is the matrix $G_i\in \F_q^{k\times m}$.
Define
\[
G^{\rm AH}:=\left(g_1\ \cdots\ g_{\frac{q^n-1}{q-1}}\right)\in \F_{q^m}^{k\times \frac{q^n-1}{q-1}}.
\]
By definition of $G^\mathrm{AH}$, we have that
$$\wt(uG^{\rm AH})=\frac{q^n-1}{q-1}-\sum_{\substack{S\in\mS_{\ccc}^{\rm AH} \\ S\subseteq  u^\perp}}\mm(S).$$
Hence,
$$\sum_{\substack{S\in\mS_{\ccc}^{\rm AH} \\ S\subseteq  u^\perp}}\mm(S)=\frac{|\{X\in \mS  \; : \; \colsp_{\F_q}(X)\subseteq  u^\perp\}|-1}{q-1}= \frac{|\mS_{\ccc}(u^\perp)|-1}{q-1},$$
so that
$$\wt(uG^{\rm AH})=\frac{q^n-|\mS_{\ccc}(u^\perp)|}{q-1}=\frac{q^n-q^{n-\rk(m_1(T,u))}}{q-1}.$$
\end{proof}

\begin{example}
Let $\mC$ be the $[3\times 3,4,2]_2$ code of the Example \ref{ex:toy}. We have already observed that the column-system $\mS$ associated with  $\mC$ is generated by
$$X_1=\begin{pmatrix}
\textcolor{black}{1}&\textcolor{black}{0}&\textcolor{black}{1}\\
\textcolor{black}{0}&\textcolor{black}{0}&\textcolor{black}{1}\\
\textcolor{black}{0}&\textcolor{black}{1}&\textcolor{black}{0}\\
\textcolor{black}{0}&\textcolor{black}{0}&\textcolor{black}{1}
\end{pmatrix}, \, 
X_2=\begin{pmatrix}
\textcolor{black}{0}&\textcolor{black}{0}&\textcolor{black}{1}\\
\textcolor{black}{0}&\textcolor{black}{0}&\textcolor{black}{1}\\
\textcolor{black}{0}&\textcolor{black}{0}&\textcolor{black}{1}\\
\textcolor{black}{0}&\textcolor{black}{1}&\textcolor{black}{1}
\end{pmatrix}, \; 
X_3=\begin{pmatrix}
\textcolor{black}{0}&\textcolor{black}{0}&\textcolor{black}{0}\\
\textcolor{black}{1}&\textcolor{black}{0}&\textcolor{black}{1}\\
\textcolor{black}{0}&\textcolor{black}{0}&\textcolor{black}{1}\\
\textcolor{black}{0}&\textcolor{black}{0}&\textcolor{black}{1}
\end{pmatrix}.
$$
We have that 
\begin{align*}
{\rm Ext}^{\rm AH}(\mS)=\{&\mathbb{P}(\colsp_{\F_2}(X_1)), \; \mathbb{P}(\colsp_{\F_2}(X_2)), \; \mathbb{P}(\colsp_{\F_2}(X_3)),\\
&\mathbb{P}(\colsp_{\F_2}(X_1+X_2)), \; \mathbb{P}(\colsp_{\F_2}(X_1+X_3)),\\&\mathbb{P}(\colsp_{\F_2}(X_2+X_3)), \; \mathbb{P}(\colsp_{\F_2}(X_1+X_2+X_3))\}.
\end{align*}
Let us choose a basis $\Gamma=\{1,\alpha,\alpha^2\}$ of $\F_8=\F_2(\alpha)$, where $\alpha^3=\alpha+1$. From this we get the $[7,4/3,6]_2^3$ code $\mC^{\rm AH}$ with generator matrix 
$$G=\begin{pmatrix}
1+\alpha^2&\alpha^2&0&1&1+\alpha^2&\alpha^2& 1\\
\alpha^2&\alpha^2&1+\alpha^2&0&1&1&1+\alpha^2\\
\alpha&\alpha^2&\alpha^2&\alpha+\alpha^2&\alpha+\alpha^2&0& \alpha\\
\alpha^2&\alpha+\alpha^2&\alpha^2&\alpha& 0& \alpha& \alpha+\alpha^2\\
\end{pmatrix}.$$
\end{example}

\begin{remark}
    A similar reasoning can be done in order to associate a row-Hamming  metric additive code to a row-nondegenerate rank-metric code. In this case, for a given code $\mC\in \FMRr$, a \emph{row-Hamming-metric code} $\mC^{\rm AH}\in \mathscr{F}_{\rm AH}(\frac{q^m-1}{q-1},k)$ \emph{associated with} $\mC$ is any code in the equivalence class $((\phi^{\rm AH})^{-1}\circ{\rm Ext_{\rrr}^{AH}}\circ\phi^{\rm MR_{\rrr}})([\mC])$, where
    \[{\rm Ext_{\rrr}^{AH}}:\SMRr/_\sim\to \mathscr{S}_{\rm AH}\left(\frac{q^m-1}{q-1},k\right)/_\sim, \ \ \ [\mT]\mapsto [\mT_{\rrr}^{\rm AH}],\]
    and 
    \[\mT_{\rrr}^{\rm AH}:=\{\mathbb{P}(\rowsp_{\F_q}(Y)): Y\in\mR(\mT)\},\]
    where $\mR(\mT)\subseteq \mT^\ast$ is a set of representatives of $\mT^\ast/\F_q^\ast$.
\end{remark}

Theorem \ref{thm:extended} shows how the fundamental parameters of a column-nondegenerate $[m\times n,k]$ rank-metric code $\mC$ relate to those of an associated Hamming-metric additive code $\mC^{\rm {AH}}$. The connection can be made even more precise. For example, we can say how the weight distributions of the two codes relate to each other.

\begin{corollary}
    Let $\mC\in \FMRc$ be a column-nondegenerate rank-metric code with rank distribution $(A_i(\mC))_{0\leq i \leq \min\{m,n\}}$. Then, the Hamming-weight distribution of a column-Hamming-metric code $\mC^{\rm AH}\in \mathscr{F}_{\rm AH}(\frac{q^n-1}{q-1},k)$ associated with  $\mC$ is $(A_j(\mC^{\rm{AH}}))_{j}$, with
    $$A_j(\mC^{\rm{AH}})= \begin{cases}
        A_i(\mC) & \textnormal{ if } j=\frac{q^n-q^{n-i}}{q-1},\\
        0 & \textnormal{ otherwise}.
    \end{cases}$$
\end{corollary}

Theorem~\ref{thm:extended} also allows to relate the generalized weights of $\mC$ and  $\mC^{\rm{AH}}$ (in the respective metrics). Indeed, those invariants can be defined for any linear code in any metric. For additive Hamming-metric codes these objects have been recently studied in \cite{d2025generalized}.

\begin{theorem}\label{thm:gen_weights_extended}
   Let $\mC\in \FMRc$ and let $\mC^{\rm AH}\in \mathscr{F}_{\rm AH}(\frac{q^n-1}{q-1},k)$ be a column-Hamming-metric code  associated with $\mC$.  For every $i \in [k]$, the $i$-th generalized Hamming weight $d_i(\mC^{\rm AH})$ and the $i$-th row-generalized weight $d_i^{\rrr}(\mC)$ are related by
    $$d_i(\mC^{\rm AH}) = \frac{q^n - q^{n - d_i^{\rrr}(\mC)}}{q - 1}.$$
\end{theorem}
\begin{proof}
    Let $\mS_{\ccc}^{\rm AH} = \phi(\mC^{\rm AH})$ be an $m$-projective system associated with $\mC^{\rm{AH}}$, i.e.,
    $$ \mS_{\ccc}^{\rm AH} \in ({\rm Ext_{\ccc}^{AH}}\circ\phi^{\rm MR_{\ccc}})([\mC]).$$
    Let $T\in\F_q^{k\times m\times n}$ be a generator tensor for $\mC$ and let $\mS$ be the column-system associated with~$T$, that is $\mS=\sss_3(T)$. Then $\mS_{\ccc}^{\rm AH}$ is equivalent to the multiset $\{\mathbb{P}(\colsp_{\F_q}(X)) : X\in\mR(\mS)\}$, where $\mR(\mS)$ is a set of representatives of $\mS^\ast/\F_q^\ast$.
    By the geometric characterization of generalized Hamming weights for additive codes (see \cite{d2025generalized}), we have
    $$\frac{q^n - 1}{q - 1} - d_i(\mC^{\rm AH}) = \max \{ |\{S \in \mS_{\ccc}^{\rm AH}  \; : \; S \subseteq \mU\}| \; : \; \mU \subseteq \F_q^k,  \; {\rm{codim}}(\mU) = i \}.$$
Since $\mS_{\ccc}^{\rm AH}$ is obtained from representatives of
$\mS^*/\F_q^*$, for every subspace $\mU\subseteq \F_q^k$ we have
\[
|\{S\in \mS_{\ccc}^{\rm AH}  \; : \; S\subseteq \mU\}|
=
\frac{|\mS_{\ccc}(\mU)|-1}{q-1}.
\]
Hence
\[
\frac{q^n-1}{q-1}-d_i(\C^{\rm AH})
=
\max_{{\rm{codim}}(\mU)=i}
\frac{|\mS_{\ccc}(\mU)|-1}{q-1}.
\]
    Moreover, by Eq.~\eqref{eq:gen-row-cwt}, we have that the $i$-th row-generalized weights of $\mC$ are then equal to
    $$d_i^{\rrr}(\mC) = n-\max\{\dim(\mS_{\ccc}(\mU))\;: \; \text{codim}(\mU) = i\}.$$
    Substituting into the previous identity gives
    $$\frac{q^n - 1}{q - 1} - d_i(\mC^{\rm AH}) = \frac{q^{n - d_i^{\rrr}(\mC)} - 1}{q - 1},$$
    and, hence, it completes the proof.
\end{proof}

\bigskip

\bibliographystyle{abbrv}
\bibliography{references}
\end{document}